\documentclass[12pt]{amsart}
\usepackage{hyperref}
\usepackage{amsmath}
\usepackage{amsthm}
\usepackage{amsfonts}
\usepackage{amssymb}
\usepackage{graphicx}
\DeclareGraphicsExtensions{.eps}
\DeclareMathOperator{\sig}{sign}
\newtheorem{theorem}{Theorem}[section]
\newtheorem{utv*}{Proposition}
\newtheorem{hyp*}{Conjecture}
\newtheorem{lemma}[theorem]{Lemma}
\newtheorem{collorary}[theorem]{Corollary}
\newtheorem{defin}{Definition}
\newtheorem{zamech}{Remark}
\newtheorem*{th*}{Theorem}

\newcommand{\av}[2]{\langle #1\rangle_{_{\scriptstyle #2}}}
\newcommand{\ave}[1]{\langle #1\rangle}
\newcommand{\eq}[2][label]{\begin{equation}\label{#1}#2\end{equation}}
\def\sli{\sum\limits}
\def\ili{\int\limits}

\def\la{\lambda}

\def\R{\mathbb{R}}

\def\vf{\varphi}

\def\Om{\Omega}

\input epsf.sty
\begin{document}

\title{Sharp weak type estimates for weights in the class $A_{p_1, p_2}$}
\author{Alexander Reznikov}
\address{Department of Mathematics, Michigan State University, East Lansing, MI 48824, USA;} \address{St.-Petersburg Department of the Steklov Mathematical Institute, Fontanka, 27, 191023, Saint Petersburg, Russia.}
\date{}
\maketitle
\begin{abstract}
We get sharp estimates for the distribution function of nonnegative weights, which satisfy so called $A_{p_1, p_2}$ condition. For particular choices of parameters $p_1$, $p_2$ this condition becomes an $A_p$-condition or Reverse H\"{o}lder condition. We also get maximizers for these sharp estimates.
We use the Bellman technique and try to carefully present and motivate our tactics.
As an illustration of how these results can be used, we deduce the following result: if a weight $w$ is in $A_2$ then it self-improves to a weight, which satisfies a Reverse H\"{o}lder condition.
\end{abstract}
\section{Introduction}
\subsection{Problem setting: basic definitions}
Put $I=[0,1]$ and take $p_{1}>p_{2}$, $p_{i}\not= 0, \pm\infty$.
For every non-negative function $\vf$ and any interval $J\subset I$ we denote
$$
\av{\vf}{J}=\frac{1}{|J|}\ili_{J}\vf (t) dt,
$$
where $|J|$ is a length of the interval $J$.
For simplicity, when we take an average over the whole interval $I$, we drop the subindex and write $\ave{\vf}$.

Take a nonnegative function $w$. Note that by the H\"{o}lder inequality we have
\begin{equation}\label{hol}
\av{w^{p_{1}}}{J}^{\frac{1}{p_{1}}}\geqslant \av{w^{p_{2}}}{J}^{\frac{1}{p_{2}}}.
\end{equation}
Let $Q>1$. We are going to consider such functions $w\geqslant 0$ that the following ``reverse'' inequality is true:
\begin{equation}\label{invhol}
\av{w^{p_{1}}}{J}^{\frac{1}{p_{1}}}\leqslant Q \cdot \av{w^{p_{2}}}{J}^{\frac{1}{p_{2}}} \qquad \forall J\subset I.
\end{equation}
If $p_{1}>p_{2}>1$ then \eqref{invhol} is called the reverse H\"{o}lder inequality. If $p_{1}=1$, $p_{2}=-\frac{1}{p-1}$ for a certain $p>1$ then \eqref{invhol} is a famous $A_{p}$-condition.

If $w\geqslant 0$ satisfies \eqref{invhol}, we write
$$
w\in A_{p_{1}, p_{2}}^{Q}.
$$
We are interested in the following question: how big can $w$ be? That is, for given $\lambda$, we want to estimate the measure of the set
$$
\{t\in I: w(t)\geqslant \lambda\}.
$$
\newpage
\subsection{Bellman setting and initial properties}
\subsubsection{History of the question}
Recently theory of weighted estimates had a great progress. $A_p$ weights play a key role in theory of singular integrals on weighted spaces. That is why we think that sharp estimates for their distribution function is interesting.

Bellman function related to harmonic analysis appeared in the work of Burkholder, ~\cite{Bu}. After that the first appearance was in the preprint of the paper by Nazarov-Treil-Volberg, ~\cite{NTV}.

Slowly different methods to find an exact Bellman function were developed. Reader can find them in papers ~\cite{SlVa}, ~\cite{VaVo}, ~\cite{Va2}, ~\cite{Va}.

There were two works by V. Vasyunin, ~\cite{Va2}, ~\cite{Va}, which are related to the question we are concerned in. He gave a sharp estimate of the $\av{w^{q}}{I}$ for every $q\in \R$, with the assumption that $w\in A_{p_1, p_2}$.
After the work ~\cite{Va2}, M. Dindo\v{s} and T. Wall, ~\cite{DiWa}, found the sharp $A_p$-``norm'' of a function, which is in a Reverse H\"{o}lder class, mentioned above.
V. Vasyunin used a Bellman technique and we shall follow it. However, we should make some changes, since in Vasyunins work he was able to reduce the question to solving a certain ODE. We can not do it and we are going to solve a PDE, following the Monge--Amp\`{e}re Technique, see ~\cite{VaVo}.

We should mention the following. After the work ~\cite{SlVa} was finished, there was an investigation of the question, similar to our, but in the space $BMO$ instead of $A_{p_1,p_2}$ (i.e., estimating the distribution function of a function $\vf$, which is in a ``ball'' in the $BMO$ ``norm''). Even though it was discussed, it was never published. We follow the pattern of this investigation.

Finally, we mention that applications of estimates we are giving arise in many questions, related to Calderon-Zygmund operators. In the Section \ref{illustr} we show how to deduce a Reverse H\"older inequality with sharp power for $A_2$ weights. Such inequalities are very useful, we refer the reader to papers ~\cite{Pe} and ~\cite{HytPe}. 
\subsubsection{Acknowledgements}
I am infinitely grateful to Professors Vasily Vasyunin and Alexander Volberg for spending a lot of time discussing this problem with me.

I also want to thank Carlos P\'{e}rez, the organizers of the 19th Summer St. Petersburg Meeting in Mathematical Analysis and the Analysis and PDE seminar at Michigan State University for an opportunity to present these results.

Finally and mostly I want to thank my mother for all her support.
\subsubsection{Initial definitions}
Denote
$$
\Om=\{x=(x_{1}, x_{2})\colon x_{i}\geqslant 0,\ x_{2}^{\frac{1}{p_{2}}}\leqslant
x_{1}^{\frac{1}{p_{1}}}\leqslant Q x_{2}^{\frac{1}{p_{2}}}\}.
$$
For every point $x\in \Om$ we set
\begin{equation}\label{defB}
\mathcal{B}(x_{1}, x_{2}; \lambda)=\sup \left\{ |\{t\colon w(t)\geqslant \lambda \}| \colon \ave{w^{p_{1}}}=x_{1}, \ave{w^{p_{2}}}=x_{2}, w\in A_{p_{1}, p_{2}}^{Q} \right \}.
\end{equation}
Note that the definition of the $\Omega$ is caused by \eqref{hol} and \eqref{invhol}.

We also need the following remark:
\begin{zamech}
For every point $x=(x_{1}, x_{2})\in \Om$ there is a function $w\in A_{p_{1}, p_{2}}^{Q}$ such that $\ave{w^{p_{1}}}=x_{1}$ and $\ave{w^{p_{2}}}=x_{2}$.
\end{zamech}
This remark has a proof, which is the same as the proof of the Lemma \ref{kusochnaya}.
This remark shows that $\mathcal{B}$ is defined (not equal to $-\infty$) on the whole domain $\Omega$.

\begin{zamech}
Obviously\textup, if $\lambda\leqslant 0$ then $\mathcal{B}(x; \lambda)=1$ for every $x$.

In the future we consider only $\lambda>0$.
\end{zamech}
Denote
$$
\Gamma_{1}=\{(x_{1}, x_{2})\colon x_{i}\geqslant 0, x_{2}^{\frac{1}{p_{2}}}=
x_{1}^{\frac{1}{p_{1}}} \},
$$
$$
\Gamma_{Q}=\{(x_{1}, x_{2})\colon x_{i}\geqslant 0, x_{1}^{\frac{1}{p_{1}}}= Q
x_{2}^{\frac{1}{p_{2}}}\}.
$$

\begin{lemma}\label{boundlambda}
Let $(v^{p_{1}}, v^{p_{2}})\in \Gamma_{1}$. Then
$$
\mathcal{B}(v^{p_{1}}, v^{p_{2}}; \lambda)=\begin{cases} 1, &v\geqslant \lambda \\
                                      0, &v<\lambda.
                                           \end{cases}
$$
\end{lemma}
\begin{proof}[Proof]
Let $\ave{w^{p_{1}}}=v^{p_{1}}$, $\ave{w^{p_{2}}}=v^{p_{2}}$. Then the H\"{o}lder inequality becomes an equality and therefore $w$ is identically equal to $v$. Thus if $v\geqslant \lambda$ then $\{t: w(t)\geqslant \lambda\}=I$ and $\mathcal{B}(v^{p_{1}}, v^{p_{2}}; \lambda)=1$. Similarly, if $v<\lambda$ then $\mathcal{B}(v^{p_{1}}, v^{p_{2}}; \lambda)=0$.
\end{proof}

Now we are going to get rid of the $\lambda$ using homogeneity. Take $w\in A_{p_{1}, p_{2}}^{Q}$ and $\ave{w^{p_{1}}}=x_{1}$, $\ave{w^{p_{2}}}=x_{2}$. For a positive number $s$ denote $\tilde{w}(t)=s\cdot w(t)$. Then $\tilde{w}\in A_{p_{1}, p_{2}}^{Q}$ and $\ave{\tilde{w}^{p_{1}}}=s^{p_{1}}x_{1}$, $\ave{\tilde{w}^{p_{2}}}=s^{p_{2}}x_{2}$.
Also
$$
w(t)\geqslant \lambda \Leftrightarrow \tilde{w}(t)\geqslant s\lambda.
$$
Therefore,
$$
\mathcal{B}(x_{1}, x_{2}; \lambda)=\mathcal{B}(s^{p_{1}}x_{1}, s^{p_{2}}x_{2}; s\lambda).
$$
Put $s=\frac{1}{\lambda}$. Then we get
$$
\mathcal{B}(x_{1}, x_{2}; \lambda)=\mathcal{B}(\lambda^{-p_{1}}x_{1}, \lambda^{-p_{2}}x_{2}; 1),
$$
so it suffices to find only $\mathcal{B}(x_{1}, x_{2}; 1)$ for every $(x_{1}, x_{2})\in \Om$.
We set
$$
\mathcal{B}(x_{1}, x_{2})=\mathcal{B}(x_{1}, x_{2}; 1).
$$
Lemma \ref{boundlambda} tells that
\begin{equation}\label{bound}
\mathcal{B}(v^{p_{1}}, v^{p_{2}})=\begin{cases} 1, &v\geqslant 1 \\ 0, &v<1. \end{cases}
\end{equation}
\subsection{Structure of the paper}
We would like to write the structure of the paper. We do it here because we have just defined the main object of the paper.

The strategy is the following: we deduce some heuristic properties of function $\mathcal{B}$. Then we try to find a function $B$, which satisfies these properties.

In subsections \ref{sectlocconc} and \ref{secdeghess} we deduce the main property of $\mathcal{B}$. Then in Section \ref{secteccalc} we make some technical calculations. In the subsection \ref{parmotiv} we present some more ideology which helps in such kind of problems.

Further, in the Section \ref{search} we show how to find an appropriate candidate for $\mathcal{B}$. The main machinery that will be used is so called Monge--Amp\`{e}re equation, which we briefly describe in Subsection \ref{MongeAmpere}. The reader can read ~\cite{VaVo} for more examples.

After we find a function $B$, which is the most natural candidate for being our Bellman function, we start proving that $B=\mathcal{B}$. In the Section \ref{bolshe} we prove that $B\geqslant \mathcal{B}$.

To prove that $B\leqslant \mathcal{B}$, we need to take an $x$, $x\in\Om$, and find a function $w\in A_{p_{1}, p_{2}}^{Q}$ such that $(\ave{w^{p_{1}}}, \ave{w^{p_{2}}})=x$ and $B(x)=|\{t\colon w(t)\geqslant 1\}|$. We shall do it in the Section \ref{menshe}. We also emphasize that the function $\mathcal{B}$ and its properties somehow carry an information about attainability of the supremum (i.e., if $\sup=\max$) and about the maximizer.

\subsection{The main property: local concavity}\label{sectlocconc}
We give the following definition.
\begin{defin}
A function $F$ is called locally concave in a domain $\Omega$ if for every $x\in \Omega$, for every convex neighborhood $U$ of $x$, such that $U\subset \Omega$, the following inequality holds\textup:
$$
F(\mu x + (1-\mu)y)\geqslant \mu F(x)+(1-\mu)F(y), \; \; \forall y\in U, \; \; \forall \mu\in[0,1].
$$
\end{defin}
In this section we will use some heuristics to conclude the main property of the Bellman function. We can not prove this property directly, but we will use it to get an appropriate candidate for $\mathcal{B}$.

Assume that we have two points $y=(y_{1}, y_{2})\in \Om$ and $z=(z_{1}, z_{2})\in\Om$, and the interval $[y,z]=\{\mu y + (1-\mu)z\colon \mu\in[0,1]\}\subset \Omega$. Assume for simplicity that the supremum in \eqref{defB} is attained on functions $w_{y}$ and $w_{z}$ respectively.
For some $\mu \in (0,1)$ take $x=\mu y + (1-\mu) z$ --- a point on the line segment, which connects $y$ and $z$. Denote
$$
w(t)=\begin{cases} w_{y}(\frac{t}{\mu}), &t\in [0, \mu) \\
                   w_{z}(\frac{t}{1-\mu}), &t\in [\mu, 1].
     \end{cases}
$$
Then
$$
\ave{w^{p_{k}}}=\ili_{0}^{\mu}w_{y}^{p_{k}}(\frac{t}{\mu})dt + \ili_{\mu}^{1} w_{z}^{p_{k}}(\frac{t}{1-\mu})dt=\mu y_{k}+ (1-\mu)z_{k}=x_{k}.
$$

To be able to compare $\mathcal{B}(x)$ with $|\{t\colon w(t)\geqslant 1\}|$ we need one more thing, namely, $w\in A_{p_{1}, p_{2}}^{Q}$. However, we can not prove it.
Nevertheless, if $w\in A_{p_{1}, p_{2}}^{Q}$ then
\begin{multline}
\mathcal{B}(x)\geqslant|\{t\colon w(t)\geqslant 1\}|=\mu\cdot |\{t\colon w_{y}(t)\geqslant 1\}| + (1-\mu) \cdot|\{t\colon w_{z}(t)\geqslant 1\}| =\\= \mu \mathcal{B}(y) + (1-\mu)\mathcal{B}(z).
\end{multline}
This property of a function is called local concavity. Note that we did not prove it.
\subsection{Degeneration of the Hessian}\label{secdeghess}
Assume that we have a smooth function $B$. Then $B$ is local concave if and only if
$$
\frac{d^{2}B}{dx^{2}}=\begin{pmatrix} B_{x_{1}x_{1}} & B_{x_{1}x_{2}} \\ B_{x_{2}x_{1}} & B_{x_{2}x_{2}} \end{pmatrix} \leqslant 0.
$$
Moreover, we want to find the ``best'' concave function. ``Best'' means that $B$ must be as small as possible (since we want to estimate something from above). It gives us a hope that local concavity is ``sharp'', i.e., that $\frac{d^{2}B}{dx^{2}}$ degenerates (as a trivial example we mention that in the one variable case a straight line is the smallest concave function with fixed boundary values). Namely, for every point $x\in \Omega$ there is a direction $\overrightarrow{m}(x)$ such that $B$ is linear in this direction. This just means that
\begin{equation}\label{det}
\det\left( \frac{d^{2}B}{dx^{2}}\right)=0.
\end{equation}

Our plan is the following. In Section \ref{search} we find a function $B(x)$, defined in $\Omega$, which is locally concave and satisfies \eqref{det} and the boundary condition \eqref{bound}.

Next, in Section \ref{bolshe} we prove that $B(x)\geqslant \mathcal{B}(x)$ and in Section \ref{menshe} we prove that $B(x)\leqslant \mathcal{B}(x)$. It will mean that $\mathcal{B}(x)=B(x)$ and, therefore, that we reached our goal.
\subsection{On the Monge--Amp\`{e}re PDE}\label{MongeAmpere}
In this subsection we state the following known result.
\begin{theorem}
Let $B$ be a function defined in $\Omega$ and assume
$$
\det\left( \frac{d^{2}B}{dx^{2}}\right)=0.
$$
Then $B$ can be represented as $B(x)=t_{0}+t_{1}x_{1}+t_{2}x_{2}$, where $t_{1}=B^{\prime}_{x_{1}}$, $t_{2}=B^{\prime}_{x_{2}}$, and
\begin{equation}\label{difform}
dt_{0}+x_{1}dt_{1}+x_{2}dt_{2}=0.
\end{equation}
\end{theorem}
This theorem can be understood in the following way. Define $t_0=B(x)-x_1 B^{\prime}_{x_1} - x_2 B^{\prime}_{x_2}$, $t_{1,2}=B^{\prime}_{x_{1,2}}$. Then for every $x\in \Om$ there is line segment, which contains $x$ and on which $t_{0}, t_1, t_2$ do not change.

\section{All technical calculations}\label{secteccalc}
In this section we shall state and proof many formulas that we will need in the future. Then we shall refer to these formulas and relations.
\subsection{Initial calculations}\label{initcalc}
We start from formalizing the geometry of $\Omega$. First, we prove the following lemma:
\begin{lemma}\label{twogammas}
For every $Q$, such that $Q>1$,  there are two solutions $\gamma_{\pm}$ $(0<\gamma_{-}<1<\gamma_{+})$ of the following equation:
\begin{equation}\label{gamma}
Q^{-p_{2}}\left(1-\frac{p_{2}}{p_{1}}\right)\gamma^{p_{2}}=
1-\frac{p_{2}}{p_{1}}Q^{-p_{2}}\gamma^{p_{2}-p_{1}}.
\end{equation}
\end{lemma}
\begin{proof}[Proof]
Put
$f(t)=\left(1-\frac{p_{2}}{p_{1}}\right)t^{p_{2}}+\frac{p_{2}}{p_{1}}t^{p_{2}-p_{1}}$.
We want to prove that there are two values of $t$ such that
$f(t)=Q^{p_{2}}$. Obviously,
$$
f^{\prime}(t)=p_{2}\frac{p_{1}-p_{2}}{p_{1}}t^{p_{2}-1}+
\frac{p_{2}}{p_{1}}(p_{2}-p_{1})t^{p_{2}-p_{1}-1}=
\frac{p_{2}}{p_{1}}(p_{1}-p_{2})t^{p_{2}-p_{1}-1}(t^{p_{1}}-1).
$$
Observe that
$$
\sig\left(\frac{t^{p_{1}}-1}{p_{1}}\right)=\sig(t-1),
$$
so
$$
\sig(f^{\prime}(t))=\sig(p_{2}(t-1)).
$$
Now we consider two cases.

Case 1: $p_{2}>0$. Then $f(0)=\infty$, $f(\infty)=\infty$ and $f(1)=1$. Moreover, when $t\in [0,1]$ then $f(t)$ decreases from $\infty$ to $1$; when $t\in[1, \infty]$ then $f(t)$ increases from $1$ to $\infty$. The observation that $Q^{p_{2}}>1$ finishes the proof for this case.
\begin{center}
\includegraphics[width=0.5\linewidth]{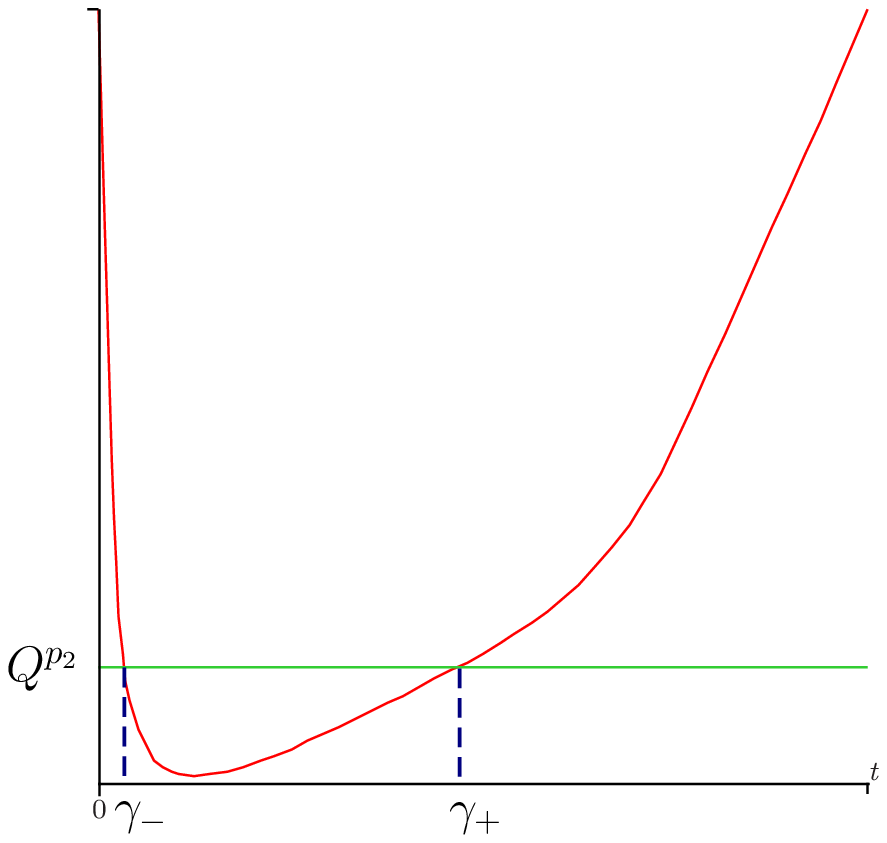}
\end{center}

Case 2: $p_{2}<0$. Then $f(0)=-\infty$, $f(\infty)=0$, $f(1)=1$. Moreover, when $t\in [0,1]$ then $f(t)$ increases from $-\infty$ to $1$; when $t\in[1, \infty]$ then $f(t)$ decreases from $1$ to $0$. The observation that $Q^{p_{2}}<1$ finishes the proof.

\begin{center}
\includegraphics[width=0.5\linewidth]{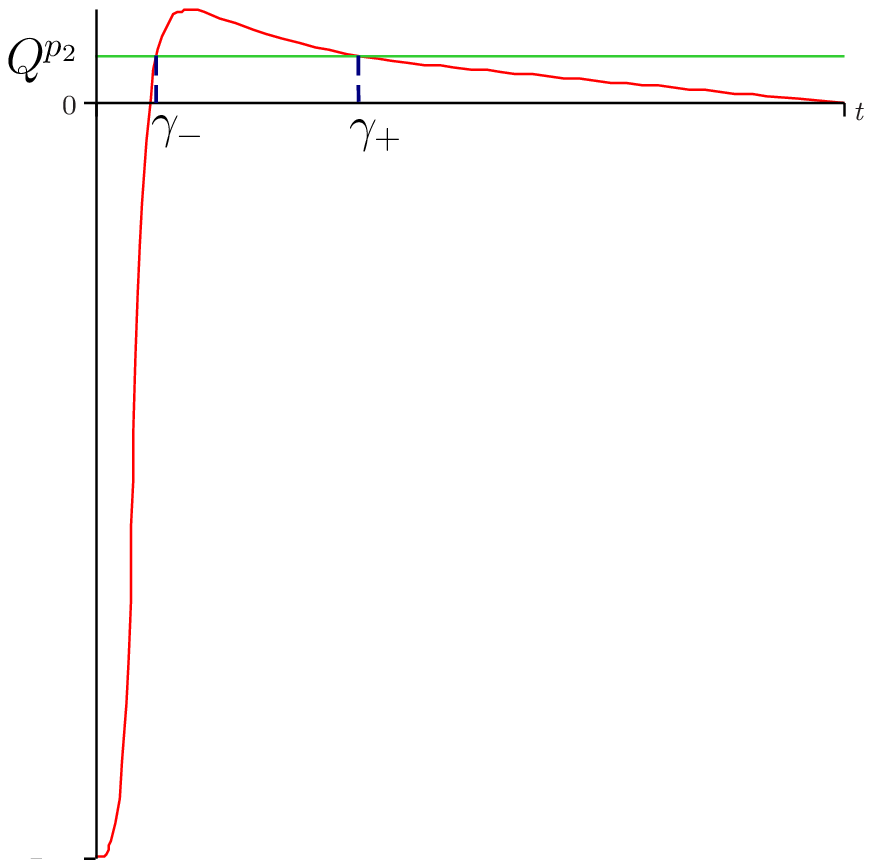}
\end{center}
\end{proof}
\begin{lemma}\label{lemma1}
For every point $(v^{p_{1}}, v^{p_{2}})\in \Gamma_1$ there are two tangent lines $\ell_{+}(v)$ and $\ell_{-}(v)$ to the $\Gamma_{Q}$, such that $v\in \ell_{\pm}(v)$.
These tangent lines are defined by following equations:
\begin{equation}\label{tangents}
x_{2}=\frac{p_{2}}{p_{1}}Q^{-p_{2}}a_{\pm}^{p_{2}-p_{1}}(x_{1}-v^{p_{1}})+v^{p_{2}},
\end{equation}
where $a_\pm = \gamma_\pm v$.
\end{lemma}
\begin{proof}[Proof]
Let $(v^{p_{1}}, v^{p_{2}})\in \Gamma_{1}$. Then
$$(a_{\pm}^{p_{1}}, Q^{-p_{2}}a_{\pm}^{p_{2}})\in \Gamma_{Q}.$$
Let $\ell_{\pm}(v)$ have an equation
$$
x_{2}=\frac{p_{2}}{p_{1}}Q^{-p_{2}}a_{\pm}^{p_{2}-p_{1}}(x_{1}-v^{p_{1}})+v^{p_{2}}.
$$
First of all, $(v^{p_{1}}, v^{p_{2}})\in \ell_{\pm}(v)$. Second,
$$
(a_{\pm}^{p_{1}}, Q^{-p_{2}}a_{\pm}^{p_{2}})\in \ell_{\pm}(v).
$$
To prove this, we use the definition of $\gamma_\pm$:
$$
Q^{-p_{2}}\left(1-\frac{p_{2}}{p_{1}}\right)\gamma_{\pm}^{p_{2}}=
1-\frac{p_{2}}{p_{1}}Q^{-p_{2}}\gamma_{\pm}^{p_{2}-p_{1}},
$$
whence
$$
Q^{-p_{2}}\left(1-\frac{p_{2}}{p_{1}}\right)a_{\pm}^{p_{2}}=
v^{p_{2}}-\frac{p_{2}}{p_{1}}Q^{-p_{2}}a_{\pm}^{p_{2}-p_{1}}v^{p_{1}},
$$
therefore
$$
Q^{-p_{2}}a_{\pm}^{p_{2}}=\frac{p_{2}}{p_{1}}Q^{-p_{2}}a_{\pm}^{p_{2}-p_{1}}(a_{\pm}^{p_{1}}-v^{p_{1}})+v^{p_{2}},
$$
which is just the required property that $(a_{\pm}^{p_{1}}, Q^{-p_{2}}a_{\pm}^{p_{2}})\in \ell_{\pm}(v)$. Also the slope of $\ell_{\pm}(v)$ is equal to the derivative of the function $x_{2}=Q^{-p_{2}}x_{1}^{\frac{p_{2}}{p_{1}}}$ at the point $(a_{\pm}^{p_{1}}, Q^{-p_{2}}a_{\pm}^{p_{2}})$, which finishes the proof.
\end{proof}
\begin{zamech}\label{zamechanie}
If $p_{1}>0$, then $\gamma_{+}^{p_{1}}>1$ and we get $\frac{a_+^{p_{1}}}{v^{p_{1}}}>1$, so $a_+^{p_{1}}>v^{p_{1}}$.
Thus, for every point $x=(x_1, x_2)$ on the segment of $\ell_{+}(v)$ with endpoints $(v^{p_{1}}, v^{p_{2}})$ and
$(a_+^{p_{1}}, Q^{-p_{2}}a_+^{p_{2}})$, we have $v^{p_{1}}\leqslant x_{1}\leqslant a_+^{p_{1}}$.

If $p_{1}<0$ then we have an inverse situation and for the same reason $a_+^{p_{1}}\leqslant x_{1} \leqslant v^{p_{1}}$.
\end{zamech}
\begin{zamech}
We remind that the point where tangent $\ell_{\pm}(v)$ touches $\Gamma_{Q}$ is $(a_{\pm}^{p_{1}}, Q^{-p_{2}}a_{\pm}^{p_{2}})$.
\end{zamech}
\begin{collorary}
Take a point $(1,1)$ and correspondent tangents $\ell_{\pm}=\ell_{\pm}(1)$. They intersect $\Gamma$ one more time at points $(v_{\pm}^{p_{1}}, v_{\pm}^{p_{2}})$. These points are defined by the following equations:
$$
v_{-}=\frac{\gamma_{-}}{\gamma_{+}},
$$
$$
v_{+}=\frac{\gamma_{+}}{\gamma_{-}}.
$$
\end{collorary}
This corollary is obvious: $\ell_{\pm}$ and $\ell_{\mp}(v_{\pm})$ are same lines, namely, these are the lines passing through the points $(1,1)$ and $(v_\pm^{p_1}, v_\pm ^{p_2})$ and being tangent to $\Gamma_Q$ at $(\gamma_{\pm}^{p_1}, Q^{-p_2}\gamma_{\pm}^{p_2})$.
\begin{lemma}
Take $x=(x_{1}, x_{2})\in \Omega$, $x\not\in \Gamma_{Q}$. Then there are two tangent to $\Gamma_{Q}$ lines which pass through $x$.
\end{lemma}
The proof of this lemma is the same as the proof of the Lemma \ref{lemma1}.

We also need the following observation.
\begin{lemma}\label{estimate}
$$
1-Q^{-p_{2}}\gamma_{+}^{p_{2}-p_{1}}>0
$$
\end{lemma}
\begin{proof}[Proof]
Since
$$
Q^{-p_{2}}\left(1-\frac{p_{2}}{p_{1}}\right)\gamma_{\pm}^{p_{2}}=
1-\frac{p_{2}}{p_{1}}Q^{-p_{2}}\gamma_{\pm}^{p_{2}-p_{1}},
$$
we get
$$
Q^{-p_{2}}\left(1-\frac{p_{2}}{p_{1}}\right)\gamma_{+}^{p_{2}}=
1+\left(1-\frac{p_{2}}{p_{1}}\right)Q^{-p_{2}}\gamma_{+}^{p_{2}-p_{1}}-Q^{-p_{2}}\gamma_{+}^{p_{2}-p_{1}}
$$
so, using $\gamma_{+}>1$, we get
$$
1-Q^{-p_{2}}\gamma_{+}^{p_{2}-p_{1}}=\frac{p_{1}-p_{2}}{p_{1}}Q^{-p_{2}}\gamma_{+}^{p_{2}}(1-\gamma_{+}^{-p_{1}})>0=(p_{1}-p_{2})Q^{-p_{2}}\gamma_{+}^{p_{2}}\frac{1-\gamma_{+}^{-p_{1}}}{p_{1}}>0.
$$
\end{proof}

We also need the following lemma.
\begin{lemma}\label{kusochnaya}
Suppose we have two positive numbers $u,v$ such that the line segment, which connects points $(u^{p_{1}}, u^{p_{2}})\in \Gamma_{1}$ and $(v^{p_{1}}, v^{p_{2}})\in \Gamma_{1}$, lies in $\Omega$. Suppose also that $\mu \in [0,1]$.
Denote
$$
w(t)=\begin{cases} u, &t\in [0, \mu) \\ v, &t\in [\mu, 1] \end{cases}.
$$
Then $w\in A_{p_{1}, p_{2}}^{Q}$.
\end{lemma}
\begin{proof}[Proof]
We take an interval $J\subset I$. If $J\subset [0,\mu]$ or $J\subset [\mu,1]$ then
$$
\av{w^{p_{1}}}{J}^{\frac{1}{p_{1}}}\av{w^{p_{2}}}{J}^{-\frac{1}{p_{2}}}=1<Q.
$$
If $J=[\alpha, \beta]$, $\alpha<\mu<\beta$ then
$$
\av{w^{p_{k}}}{J}=\frac{u^{p_{k}}(\mu-\alpha)+v^{p_{k}}(\beta-\mu)}{\beta-\alpha}.
$$
It means that the point $x=(\av{w^{p_{1}}}{J}, \av{w^{p_{2}}}{J})$ is a convex combination of the points $(u^{p_{1}}, u^{p_{2}})$ and $(v^{p_{1}}, v^{p_{2}})$, so $x\in \Omega$. Thus, $x_{1}^{\frac{1}{p_{1}}}x_{2}^{-\frac{1}{p_{2}}}\leqslant Q$, and, therefore, $w\in A_{p_{1}, p_{2}}^{Q}$.
\end{proof}
\begin{zamech}
In particular, if $J=I$, then we get that $\ave{w^{p_{k}}}=\mu u^{p_{k}}+(1-\mu)v^{p_{k}}$.
\end{zamech}
\subsection{Splitting of $\Omega$: formulas}
Now we want to split $\Omega$ into different subdomains. We write precise formulas and then present a picture to show what really happens.

\subsubsection{Case 1. $p_{1}>p_{2}>0$}
\begin{align*}
\Omega_{\textup I}&=\{ x\in \Om, \; x_{2}>Q^{-p_{2}}\frac{p_{2}}{p_{1}}\gamma_{+}^{p_{2}-p_{1}}(x_{1}-1)+1 \}\cup \{x\in \Om, \; x_{1}>\gamma_{+}^{p_{1}} \},\\
\Omega_{\textup{II}}&=\{x\in \Om \; x_{2}<Q^{-p_{2}}\frac{p_{2}}{p_{1}}\gamma_{+}^{p_{2}-p_{1}}(x_{1}-1)+1,\; x_{2}<Q^{-p_{2}}\frac{p_{2}}{p_{1}}\gamma_{-}^{p_{2}-p_{1}}(x_{1}-1)+1, \\ & \; \; \; \; \; \; \; \; \; \; \; \; \; \; \;\;\;\;\;\;\;\;\;\;\;\;\;\;\;\;\;\;\;\;\;\;\;\;\;\;\;\;\;\;\;\;\;\;\;\;\;\;\;\;\;\;\;\;\;\;\;\;\;\;\gamma_{-}^{p_{1}}<x_{1}<\gamma_{+}^{p_{1}} \},\\
\Omega_{\textup{III}}&=\{x\in \Om, \;x_{2}>Q^{-p_{2}}\frac{p_{2}}{p_{1}}\gamma_{-}^{p_{2}-p_{1}}(x_{1}-1)+1\},\\
\Omega_{\textup{IV}}&=\{x\in \Om, \;x_{2}<Q^{-p_{2}}\frac{p_{2}}{p_{1}}\gamma_{-}^{p_{2}-p_{1}}(x_{1}-1)+1, \; 0<x_{1}<\gamma_{-}^{p_{1}}\}.
\end{align*}
\begin{center}
\includegraphics[width=1\linewidth]{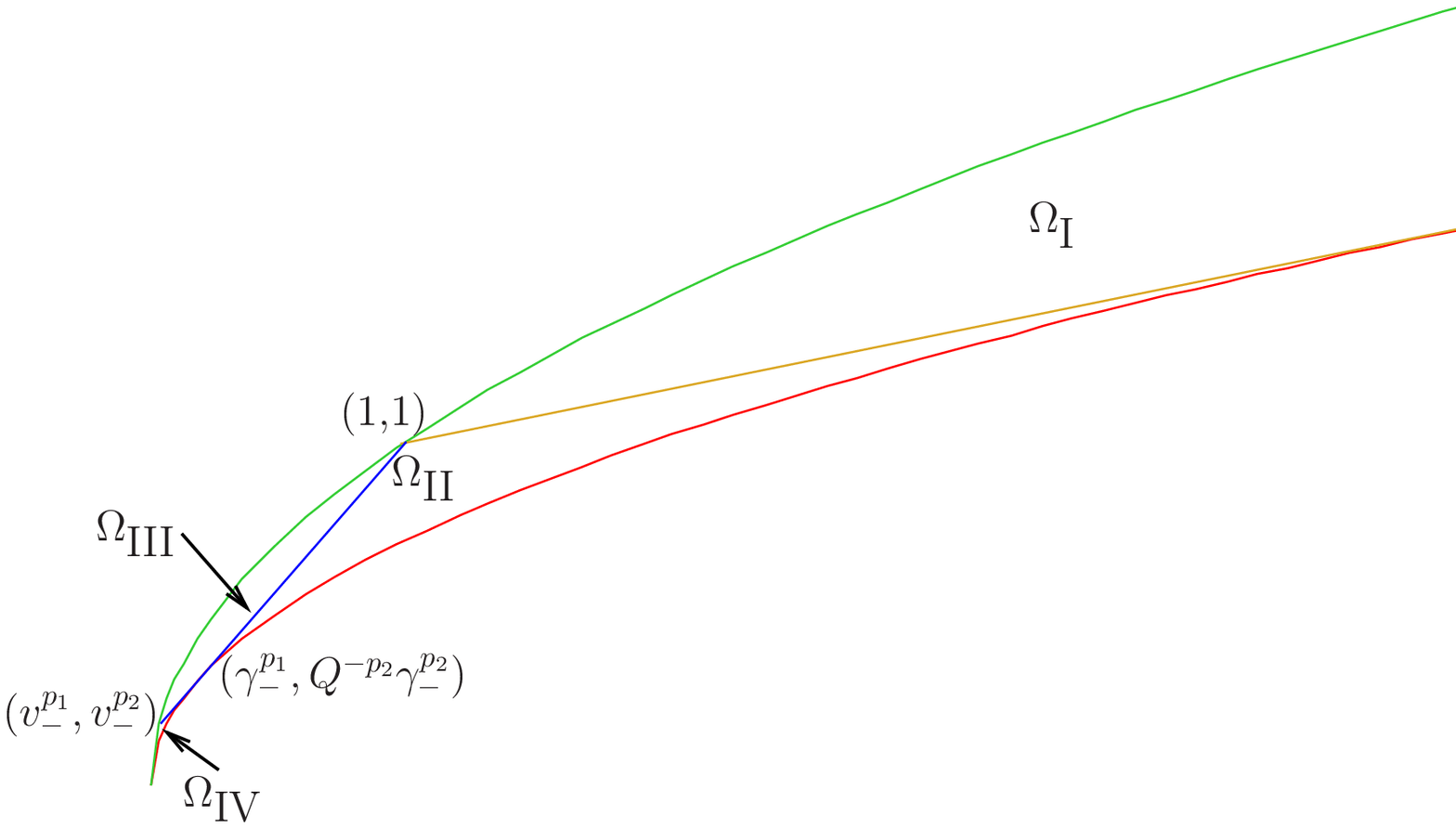}
\end{center}

\subsubsection{Case 2. $p_{1}>0>p_{2}$.}
\begin{align*}
\Omega_{\textup I}&=\{ x\in \Om, \; x_{2}<Q^{-p_{2}}\frac{p_{2}}{p_{1}}\gamma_{+}^{p_{2}-p_{1}}(x_{1}-1)+1 \}\cup \{x\in \Om, \; x_{1}>\gamma_{+}^{p_{1}} \}, \\
\Omega_{\textup{II}}&=\{x\in \Om, \;x_{2}>Q^{-p_{2}}\frac{p_{2}}{p_{1}}\gamma_{+}^{p_{2}-p_{1}}(x_{1}-1)+1,\; x_{2}>Q^{-p_{2}}\frac{p_{2}}{p_{1}}\gamma_{-}^{p_{2}-p_{1}}(x_{1}-1)+1, \\ & \; \; \; \; \; \; \; \; \; \; \; \; \; \; \;\;\;\;\;\;\;\;\;\;\;\;\;\;\;\;\;\;\;\;\;\;\;\;\;\;\;\;\;\;\;\;\;\;\;\;\;\;\;\;\;\;\;\;\;\;\;\;\;\; \gamma_{-}^{p_{1}}<x_{1}<\gamma_{+}^{p_{1}} \}, \\
\Omega_{\textup{III}}&=\{x\in \Om, \;x_{2}<Q^{-p_{2}}\frac{p_{2}}{p_{1}}\gamma_{-}^{p_{2}-p_{1}}(x_{1}-1)+1\},\\
\Omega_{\textup{IV}}&=\{x\in \Om, \;x_{2}>Q^{-p_{2}}\frac{p_{2}}{p_{1}}\gamma_{-}^{p_{2}-p_{1}}(x_{1}-1)+1, \; 0<x_{1}<\gamma_{-}^{p_{1}}\}.
\end{align*}
\begin{center}
\includegraphics[width=1\linewidth]{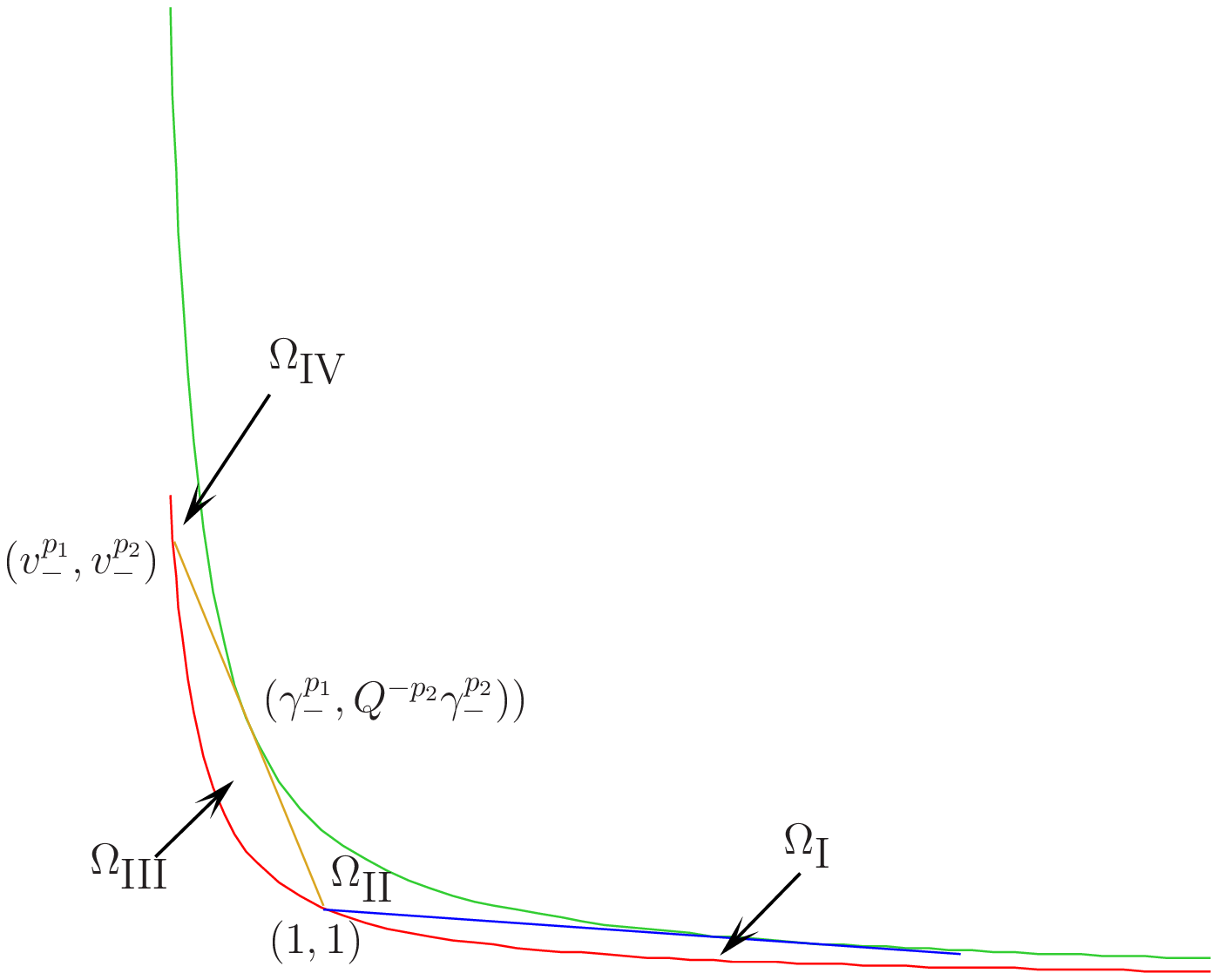}
\end{center}
\subsubsection{Case 3. $0>p_{1}>p_{2}$.}
\begin{align*}
\Omega_{\textup I}&=\{ x\in \Om, \; x_{2}<Q^{-p_{2}}\frac{p_{2}}{p_{1}}\gamma_{+}^{p_{2}-p_{1}}(x_{1}-1)+1 \}\cup \{x\in \Om, \; x_{1}<\gamma_{+}^{p_{1}} \}\\
\Omega_{\textup{II}}&=\{x\in \Om, \;x_{2}>Q^{-p_{2}}\frac{p_{2}}{p_{1}}\gamma_{+}^{p_{2}-p_{1}}(x_{1}-1)+1,\; x_{2}>Q^{-p_{2}}\frac{p_{2}}{p_{1}}\gamma_{-}^{p_{2}-p_{1}}(x_{1}-1)+1, \\ & \; \; \; \; \; \; \; \; \; \; \; \; \; \; \;\;\;\;\;\;\;\;\;\;\;\;\;\;\;\;\;\;\;\;\;\;\;\;\;\;\;\;\;\;\;\;\;\;\;\;\;\;\;\;\;\;\;\;\;\;\;\;\;\; \gamma_{+}^{p_{1}}<x_{1}<\gamma_{-}^{p_{1}} \},\\
\Omega_{\textup{III}}&=\{x\in \Om, \;x_{2}<Q^{-p_{2}}\frac{p_{2}}{p_{1}}\gamma_{-}^{p_{2}-p_{1}}(x_{1}-1)+1\},\\
\Omega_{\textup{IV}}&=\{x\in \Om,
\;x_{2}>Q^{-p_{2}}\frac{p_{2}}{p_{1}}\gamma_{-}^{p_{2}-p_{1}}(x_{1}-1)+1,\;
x_{1}>\gamma_{-}^{p_{1}}\}.\\
\end{align*}
\begin{center}
\includegraphics[width=1\linewidth]{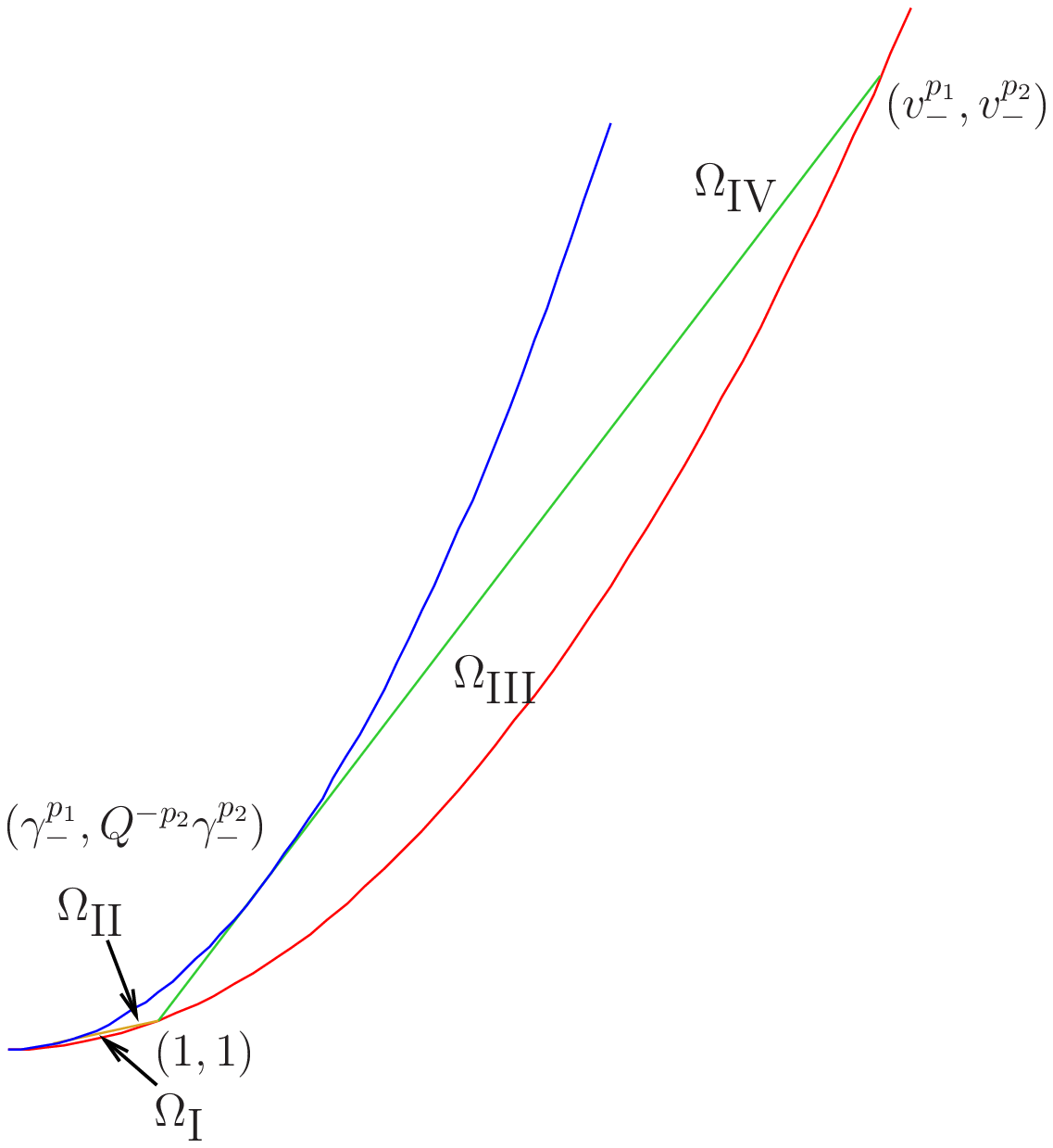}
\end{center}
%



\subsubsection{Motivation}\label{parmotiv}
We shall try to give some motivation why we choose these subdomains.

Let us make one heuristic observation. Assume we do something similar to what we did in the subsection \ref{sectlocconc}. Note that we want to make functions $w_{y}$ and $w_{z}$ as big as possible. Also note that we know extremal functions for every point on $\Gamma$, since there is only one acceptable function. Therefore, for given $x$, it is good to connect it with some point $(v^{p_{1}}, v^{p_{2}})\in\Gamma$ such that $v\geqslant 1$.

Take now $\Om_{\textup{I}}$. Geometrically, it consists of points $x$ with the following property: there exist two numbers $u,v\geqslant 1$ such that $x$ lies on the line segment, which connects $(u^{p_{1}}, u^{p_{2}})$ with $(v^{p_{1}}, v^{p_{2}})$, and this line segment lies in $\Om$. So it is a very good domain for us and we separate it.

Take now $\Om_{\textup{III}}$. Here we can not put $x$ on a line which connects two ``big'' points, but, however, we can connect $x$ with the point $(1,1)$, which is not bad, and with another point on $\Gamma$. That is why we separate this domain.

Note that for $x\in \Om_{\textup{II}}$ we can connect it with $(1,1)$ as well, but another end of the corresponding line segment will come to $\Gamma_{Q}$, which is bad, since we have no idea, which function will be extremal there.

Now we are left with $\Om_{\textup{II}}$ and $\Om_{\textup{IV}}$, and we do not split them anymore, because we do not have another motivation.
\subsection{On the dependence of $v$ on $x$}
In this section we denote a function $v$ as a function of $x$. Here is the definition.

If $x\in \Om_{\textup{III}}$ then take $v\not=1$ as a solution of the following equation:
$$
v^{p_{2}}(1-x_{1})-v^{p_{1}}(1-x_{2})=x_{2}-x_{1}.
$$
The geometrical meaning of this is easy. We take a line segment, which connects the point $x$ and the point $(1,1)$. We continue this line until it intersects $\Gamma$ one more time. The point of intersection is exactly $(v^{p_{1}}, v^{p_{2}})$. Note that for this $v$ we have $\sig(x_{1}-v^{p_{1}})=\sig(p_{1})$.

Next, take all $(v^{p_1}, v^{p_2})$ on a boundary of $\Om_{\textup{IV}}$ and a tangent line $\ell_{+}(v)$. These tangent lines cover all $\Om_{\textup{IV}}$, and they do not intersect. Therefore, for every $x\in \Om_{\textup{IV}}$ we
can find exactly one such $\ell_{+}(v)$. Notice that $\sig(x_{1}-v^{p_{1}})=\sig(p_{1})$.

Take this $\ell_{+}(v)$:
$$
x_{2}=Q^{-p_{2}}\frac{p_{2}}{p_{1}}\gamma_{+}^{p_{2}-p_{1}}v^{p_{2}-p_{1}}(x_{1}-v^{p_{1}})+v^{p_{2}}.
$$
Fix $x_{2}$ and consider $v$ as a function of $x_{1}$. We would like to find out the sign of $v^{\prime}_{x_{1}}$.
\begin{defin}
Take $v(x)$ as a solution of the following equation:
$$
\begin{cases} v^{p_{2}}(1-x_{1})-v^{p_{1}}(1-x_{2})=x_{2}-x_{1}, &x\in \Om_{\textup{III}} \\
              x_{2}=Q^{-p_{2}}\frac{p_{2}}{p_{1}}\gamma_{+}^{p_{2}-p_{1}}v^{p_{2}-p_{1}}(x_{1}-v^{p_{1}})+v^{p_{2}}, &x\in \Om_{\textup{IV}}
\end{cases},
$$
and such that $\sig(x_{1}-v^{p_{1}})=\sig(p_{1})$.
\end{defin}
Our goal is the following lemma:
\begin{lemma}
For every $x$ the following is true: $\sig(v^{\prime}_{x_{1}})=-\sig(p_{1})$.
\end{lemma}
\paragraph{Case 1: $x\in \Om_{\textup{IV}}$.}

\begin{lemma}\label{depend}
Take $x\in \Om_{\textup{IV}}$. Denote
$$
A=Q^{-p_{2}}\gamma_{+}^{p_{2}-p_{1}},
$$
$$
\Pi=\frac{Ax_{1}}{v^{p_{1}+1}}-\frac{x_{2}}{v^{p_{2}+1}}.
$$
Then
$$
v^{\prime}_{x_{1}}=\frac{1}{p_{1}\Pi}\cdot\frac{A}{v^{p_{1}}},
$$
$$
v^{\prime}_{x_{2}}=-\frac{1}{p_{2}\Pi}\cdot\frac{1}{v^{p_{2}}}.
$$
Moreover, we always have $\Pi<0$ and therefore
$$
\sig(v^{\prime}_{x_{1}})=-\sig(p_{1}), \qquad \qquad \sig(v^{\prime}_{x_2})=\sig(p_2).
$$
\end{lemma}
\begin{proof}[Proof]
Using the definition of $A$, rewrite the equation for $v$ in the following form:
$$
x_{2}=\frac{p_{2}}{p_{1}}Av^{p_{2}-p_{1}}(x_{1}-v^{p_{1}})+v^{p_{2}}=\frac{p_{2}}{p_{1}}Av^{p_{2}-p_{1}}x_{1}-\frac{p_{2}}{p_{1}}Av^{p_{2}}+v^{p_{2}},
$$
so
$$
\frac{x_{2}}{v^{p_{2}}}=\frac{p_{2}}{p_{1}}A\frac{x_{1}}{v^{p_{1}}}-\left(\frac{p_{2}}{p_{1}}A-1\right).
$$
Now let us take the partial derivative $\frac{\partial}{\partial x_{1}}$
$$
-p_{2}\frac{x_{2}}{v^{p_{2}+1}}v^{\prime}_{x_{1}}=\frac{p_{2}}{p_{1}}A\left(\frac{1}{v^{p_{1}}}-\frac{p_{1}x_{1}}{v^{p_{1}+1}}v^{\prime}_{x_{1}}\right),
$$
therefore,
$$
v^{\prime}_{x_{1}}\left(\frac{Ax_{1}}{v^{p_{1}+1}}-\frac{x_{2}}{v^{p_{2}+1}}\right)=\frac{A}{p_{1}v^{p_{1}}}.
$$
One can get the result for $v^{\prime}_{x_{2}}$ similarly.
From the equation
$$
\frac{x_{2}}{v^{p_{2}}}=\frac{p_{2}}{p_{1}}A\frac{x_{1}}{v^{p_{1}}}-\left(\frac{p_{2}}{p_{1}}A-1\right).
$$
we get that
$$
\frac{Ax_{1}}{v^{p_{1}}}-\frac{x_{2}}{v^{p_{2}}}=\frac{Ax_{1}}{v^{p_{1}}}(1-\frac{p_{2}}{p_{1}})+(\frac{p_{2}}{p_{1}}A-1).
$$
From \eqref{gamma} we see that
$$
1-\frac{p_{2}}{p_{1}}A=Q^{-p_{2}}(1-\frac{p_{2}}{p_{1}})\gamma_{+}^{p_{2}},
$$
so
\begin{multline}
$$
\frac{Ax_{1}}{v^{p_{1}}}-\frac{x_{2}}{v^{p_{2}}}=\left(1-\frac{p_{2}}{p_{1}}\right)\left(\frac{Q^{-p_{2}}\gamma_{+}^{p_{2}-p_{1}}x_{1}}{v^{p_{1}}}-Q^{-p_{2}}\gamma_{+}^{p_{2}}\right)=\\
= Q^{-p_{2}}\gamma_{+}^{p_{2}}\left(1-\frac{p_{2}}{p_{1}}\right)\left(\frac{x_{1}}{(\gamma_{+}v)^{p_{1}}}-1\right).
$$
\end{multline}
Observe that $\gamma_{+}v=a_{+}$ and from Remark \ref{zamechanie} we know that $\sig(x_{1}-a_+^{p_{1}})=-\sig(p_{1})$. Therefore, the equation
$$
p_{1}\Pi=\frac{p_{1}}{v}\left(\frac{Ax_{1}}{v^{p_{1}}}-\frac{x_{2}}{v^{p_{2}}}\right)=\frac{1}{v}Q^{-p_{2}}\gamma_{+}^{p_{2}}(p_{1}-p_{2})\left(\frac{x_{1}}{a_+^{p_{1}}}-1\right)
$$
finishes the proof.
\end{proof}
\begin{zamech}
One can draw a picture, take a point $x$, move it a little bit to the right and observe that $v^{p_{1}}$ decreased. It exactly means that $v$ acts as said in Lemma. We give a picture for the case $p_1>p_2>0$.
\end{zamech}
\begin{center}
\begin{center}
\includegraphics[width=0.9\linewidth]{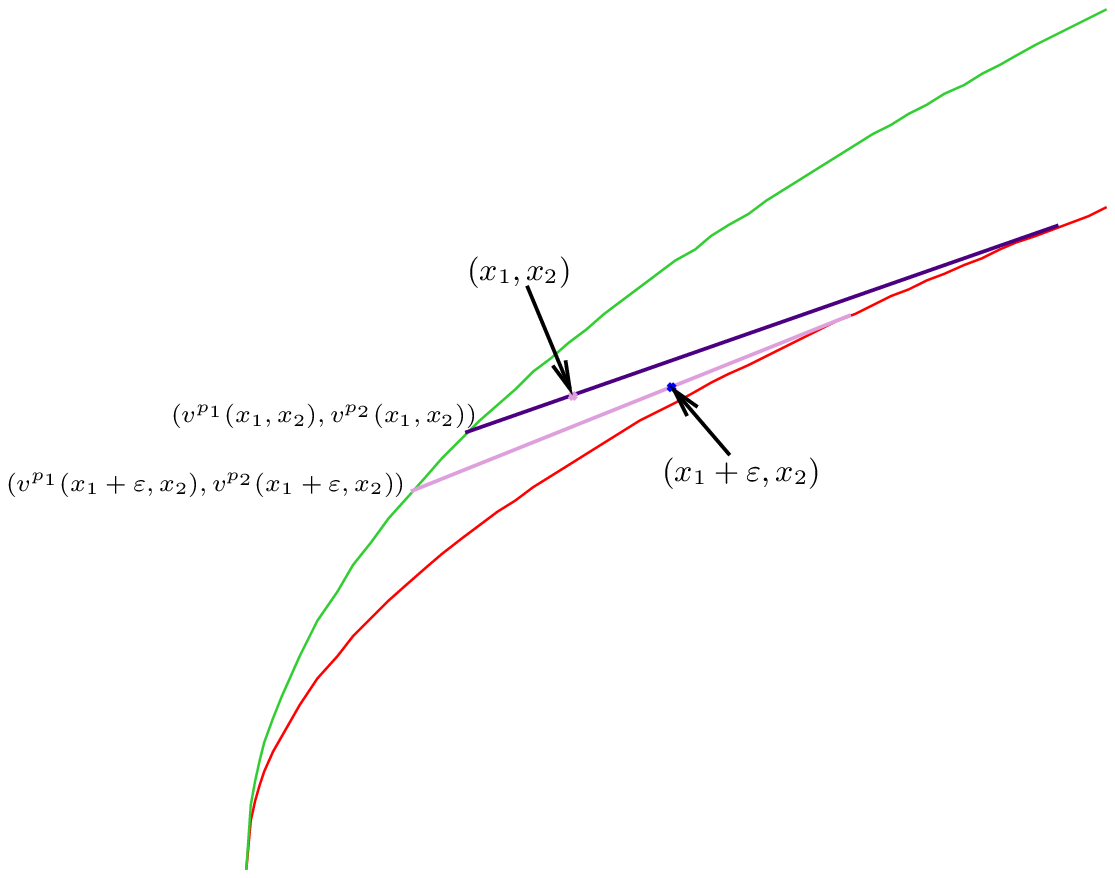}
\end{center}
\end{center}
\begin{zamech}
We would like to warn the reader about one thing: sometimes it is not always sufficient to consider a picture only for $p_{1}>p_{2}>0$. This picture was just an example of what one should draw to be convinced that our proposition is true.
\end{zamech}
\paragraph{Case 2: $x\in\Om_{\textup{III}}$.}
We would like to do the same thing as before: study the sign of $v^{\prime}_{x_{1}}$. Note that it is obvious from the picture that $v^{\prime}_{x_{1}}$ behaves absolutely in the same way as before. But we shall prove it analytically.
\begin{lemma}
Let $v$ be as above, $x\in \Om_{\textup{III}}$. Then $\sig{v^{\prime}}_{x_{1}}=-\sig{p_{1}}$.
\end{lemma}
\begin{proof}[Proof]
We do the following trick. Obviously,
$$
x_{1}-1=\frac{v^{p_{1}}-1}{v^{p_{2}}-1}(x_{2}-1).
$$
Therefore, $1=(x_2-1)h^{\prime}(v) v^{\prime}_{x_1}$, where
$$
h(v)=\frac{v^{p_{1}}-1}{v^{p_{2}}-1}.
$$
Deriving, we get
\begin{multline*}
h^{\prime}(v)=\frac{p_{1}v^{p_{1}-1}(v^{p_{2}}-1)-p_{2}v^{p_{2}-1}(v^{p_{1}}-1)}{(v^{p_{2}}-1)^{2}}=\\=\frac{v^{p_{1}+p_{2}-1}}{(v^{p_{2}}-1)^{2}}(p_{1}-p_{2}+p_{2}v^{-p_{1}}-p_{1}v^{-p_{2}})=\frac{v^{p_{1}+p_{2}-1}}{(v^{p_{2}}-1)^{2}}h_{1}(v),
\end{multline*}
where
$$
h_{1}(v)=p_{1}-p_{2}+p_{2}v^{-p_{1}}-p_{1}v^{-p_{2}},
$$
thus
$$
\sig(v^{\prime}_{x_1}) = \sig(x_2-1) \sig(h_1(v)).
$$
Clearly, $\sig(x_2-1)=-\sig(p_2)$, and we only need to find $\sig(h_1(v))$. Since
$$
h_{1}^{\prime}(v)=-p_{2}p_{1}v^{-p_{1}-1}+p_{1}p_{2}v^{-p_{2}-1}=p_{1}p_{2}v^{-p_{2}-1}(1-v^{p_{2}-p_{1}}),
$$
we have $\sig(h_1^{\prime})=-\sig(p_1 p_2)$. Note that $h_1(1)=0$, whence $\sig(h_{1}(v))=\sig(p_{1}p_{2})$. Therefore, $\sig(v^{\prime}_{x_{1}})=-\sig(p_{1})$.

%
%
%
%
%
%
\end{proof}
\subsection{On the local concavity and derivatives in the sense of generalized functions.}\label{special}
In this paragraph we are going to discuss the following question. Assume $B$ is not smooth, but still locally concave. How to express it in the sense of derivatives? The answer is easy: we must demand $\frac{d^{2}B}{dx^{2}}\leqslant 0$ in the sense of generalized functions. More precisely, the following theorem is true:
\begin{theorem}
Function $B$ is locally concave in $\Om$ if and only if for every smooth function $\vf\geqslant 0$ with a compact support in the interior of $\Om$, and for every $\Delta_{1}, \Delta_{2}\in \R$ the following inequality holds:
$$
\ili B(x)\left[\vf^{\prime\prime}_{x_{1}x_{1}}\Delta_{1}^{2}+2\vf^{\prime\prime}_{x_{1}x_{2}}\Delta_{1}\Delta_{2}+\vf^{\prime\prime}_{x_{2}x_{2}}\Delta_{2}^{2}\right]dx\leqslant 0.
$$
\end{theorem}

Our next procedure is the following: we take one of the integrals above, for example $\ili B(x)\vf^{\prime\prime}_{x_{1}x_{1}}(x)dx$, and perform an integration by parts. While doing that, we assume $B$ is continuous and $B_{x_{1}}$ is not, but we assume that it consists of two functions, which are differentiable.

Our motivation is the following: we are going to find $B$ which is twice differentiable in interiors of $\Om_{\textup{I}}, \Om_{\textup{II}}, \Om_{\textup{III}}$ and $\Om_{\textup{IV}}$, and which is continuous through $\ell_{\pm}$. However, first derivatives of $B$ will not be continuous in $\ell_{\pm}$, and we want to catch the influence of its jump on the integral above.

We state the following lemma, where $F$ plays the role of a derivative of $B$:
\begin{lemma}\label{specialfunctions}
Let $F(x_{1}, x_{2})=\begin{cases} f_{1}(x_{1}, x_{2}), &x_{2}\geqslant kx_{1}+m \\ f_{2}(x_{1}, x_{2}), &x_{2}< kx_{1}+m\end{cases}$. Let $\varphi$ be a smooth function with compact support. By $(f, \vf)$ we denote the action of the functional $f$ on the function $\vf$. Then, considering $F_{x_{1}}$ and $F_{x_{2}}$ as generalized functions, we get
$$
(F_{x_{2}}, \varphi)=\ili\ili F_{x_{2}}\varphi dx_{1}dx_{2} + \ili_{\R}(f_{1}(x_{1}, kx_{1}+m)-f_{2}(x_{1}, kx_{1}+m))\varphi(x_{1}, kx_{1}+m)dx_{1},
$$
and
$$
(F_{x_{1}}, \varphi)=\ili\ili F_{x_{1}}\varphi dx_{1}dx_{2}+k\ili_{\R} (f_{2}(x_{1}, kx_{1}+m)-f_{1}(x_{1}, kx_{1}+m))\varphi(x_{1}, kx_{1}+m)dx_{1}.
$$
\end{lemma}
Proof of this lemma is a pure integration by parts and we avoid it.

We also make the following obvious remark.
\begin{zamech}
Take $F$ as above and assume $f_{1}(x_{1}, kx_{1}+m)$ is a constant $f_{1}$ and $f_{2}(x_{1}, kx_{1}+m)$ is a constant $f_{2}$. Then
$$
(F_{x_{1}}, \vf)=\ili F_{x_{1}}\vf dx + k(f_{2}-f_{1})\ili \vf(x_{1}, kx_{1}+m)dx_{1};
$$
$$
(F_{x_{2}}, \vf)=\ili F_{x_{2}}\vf dx +(f_{1}-f_{2})\ili \vf(x_{1}, kx_{1}+m)dx_{1}.
$$
\end{zamech}

Finally, denote $\Phi_{\pm}(\vf)=\ili \vf(x_{1}, Q^{-p_{2}}\frac{p_{2}}{p_{1}}\gamma_{\pm}^{p_{2}-p_{1}}(x_{1}-1)+1)dx_{1}$. We simply integrate $\vf$ over $\ell_{\pm}$. Note that if $\vf\geqslant 0$ then $\Phi_{\pm}(\vf)\geqslant 0$.

We are going to express the derivatives of $B$ in terms of functionals $\Phi_{\pm}$ and in terms of not generalized derivatives.

To simplify our calculations we shall state a technical lemma. It says that to check that the ``jump'' matrix is nonpositive it is sufficient to check that the derivative is nonpositive only for one direction, which is not parallel to $\ell_{\pm}$.
\begin{lemma}
Let $B$ be as above. To check that the Hessian is a nonpositive general function it is sufficient to check that in interiors of $\Om_{\textup{I}}$---$\Om_{\textup{IV}}$ the Hessian is nonpositive and that $B_{x_{2}x_{2}}\leqslant 0$ as a generalized function.
\end{lemma}
The proof is obvious and based on the fact that $B^{\prime\prime}_{x_{1}x_{2}}=B^{\prime\prime}_{x_{2}x_{1}}$ and $\det \frac{d^{2}B}{dx^{2}}=0$.
\subsection{On the approximation of $A_{p_{1}, p_{2}}^{Q}$-weights with bounded weights}\label{subsectionapproxim}
In this subsection we are going to prove two results about the approximation of $A_{p_{1}, p_{2}}^{Q}$-weights. The motivation is the following: if we have an integral of a function over a finite interval, it may be convenient to ``make'' the function to be bounded, because to bounded functions we can apply the Lebesgue Dominated Convergence Theorem.

We would like to prove following lemmata.
\begin{lemma}\label{cutoff}
Assume $w\in A_{p_{1}, p_{2}}^{Q}$. Take
$$
\underline{w}_{a}(t)=\begin{cases} w(t), &w(t)\leqslant a \\
                        a, &w(t)>a
                        \end{cases}.
$$
Then $\underline{w}_{a}\in A_{p_{1}, p_{2}}^{Q}$.
The same is true for the function
$$
\overline{w}_{a}(t)=\begin{cases} a, &w(t)\leqslant a \\
                        w(t), &w(t)>a
                        \end{cases}.
$$
\end{lemma}
\begin{zamech}
Note that
$$
\underline{w}_{a}\leqslant w \leqslant \overline{w}_{a}.
$$
\end{zamech}
\begin{zamech}
Note that it is sufficient to prove the lemma only for $\underline{w}_{a}$. The result for $\overline{w}_{a}$ will follow immediately, since instead of $w, p_{1}, p_{2}$ we can consider $w^{-1}, -p_{2}, -p_{1}$.
\end{zamech}
\begin{proof}[Proof]
First, we fix an interval $J\subset[0,1]$ and denote $J_{1}=\{t\in J\colon w(t)\leqslant a\}$, $J_{2}=\{t\in J\colon w(t)>a\}$. Denote also
$$
z_i=\av{w^{p_1}}{J_i},\qquad y_i=\av{w^{p_2}}{J_i},\qquad\alpha_i=\frac{|J_i|}{|J|}\,.
$$
Then we want to prove
\begin{equation}
\label{phi}
\begin{aligned}
\av{&w^{p_1}}J^{\frac1{p_1}}\av{w^{p_2}}J^{-\frac1{p_2}}- \av{w_a^{p_1}}J^{\frac1{p_1}}\av{w_a^{p_2}}J^{-\frac1{p_2}}
\\
&=(\alpha_1z_1+\alpha_2z_2)^{\frac1{p_1}} (\alpha_1y_1+\alpha_2y_2)^{-\frac1{p_2}}-
(\alpha_1z_1+\alpha_2a^{p_1})^{\frac1{p_1}} (\alpha_1y_1+\alpha_2a^{p_2})^{-\frac1{p_2}}
\geqslant0\,.
\end{aligned}
\end{equation}

By H\"older's inequality, we get $z_i^{\frac1{p_1}}\geqslant y_i^{\frac1{p_2}}$. Therefore, if we denote $y_2^{\frac1{p_2}}$ by $u$, then $z_2^{\frac1{p_1}}=su$ for a number $s\geqslant 1$ and expression~\eqref{phi}, which we need to estimate, can be written as the following function of $s$ and $u$:
$$
\vf(s,u)=(\alpha_1z_1+\alpha_2s^{p_1}u^{p_1})^{\frac1{p_1}} (\alpha_1y_1+\alpha_2u^{p_2})^{-\frac1{p_2}}-
(\alpha_1z_1+\alpha_2a^{p_1})^{\frac1{p_1}} (\alpha_1y_1+\alpha_2a^{p_2})^{-\frac1{p_2}}\,.
$$
Since
$$
\frac{\partial\vf}{\partial s}=\alpha_2 s^{p_1-1}u^{p_1} (\alpha_1z_1+\alpha_2s^{p_1}u^{p_1})^{\frac1{p_1}-1}\geqslant0\,,
$$
the function $\vf$ is increasing in $s$ and therefore $\vf(s,u)\geqslant\vf(1,u)$, i.e., it has the minimal value when $w(t)$ is equal to $u$ on $J_2$ identically.

Now we have $u=w(t)|_{J_2}>a$ and since $\vf(1,a)=0$, the desired inequality will be proved after checking that $\frac{\partial\vf}{\partial u}(1,u)\geqslant0$. We write

\begin{align*}
\frac{\partial\vf}{\partial u}&(1,u)
\\
&=\alpha_2u^{-1}(\alpha_1z_1+\alpha_2u^{p_1})^{\frac1{p_1}-1} (\alpha_1y_1+\alpha_2u^{p_2})^{-\frac1{p_2}-1}\times
\\
&\qquad\qquad\qquad\qquad\qquad\qquad\qquad \times\big[u^{p_1}(\alpha_1y_1+\alpha_2u^{p_2})- u^{p_2}(\alpha_1z_1+\alpha_2u^{p_1})\big]
\\
&=\alpha_1\alpha_2u^{-1}(\alpha_1z_1+\alpha_2u^{p_1})^{\frac1{p_1}-1} (\alpha_1y_1+\alpha_2u^{p_2})^{-\frac1{p_2}-1}[u^{p_1}y_1-u^{p_2}z_1],
\end{align*}
and we are done because $u^{p_1}y_1-u^{p_2}z_1\geqslant0$. Indeed,
since $u\geqslant w(t)$ and $p_1\geqslant p_2$, we have $u^{p_1-p_2}\geqslant w(t)^{p_1-p_2}$, whence $u^{p_1}w^{p_2}\geqslant u^{p_2}w^{p_1}$. Therefore,
$$
u^{p_1}y_1-u^{p_2}z_1=\av{u^{p_1}w^{p_2}-u^{p_2}w^{p_1}}{J_1}\geqslant0\,,
$$
what completes the proof.
\end{proof}

\section{Searching for $B$}\label{search}
\subsection{Domain $\Omega_{\textup I}$}\label{subsectionomega1}
\begin{lemma}\label{om1conv}
For every point $x=(x_{1}, x_{2})\in \Om_{\textup{I}}$ there are two numbers $u\geqslant 1$ and  $v\geqslant 1$ such that $x$ lies on the line segment which connects $(u^{p_{1}}, u^{p_{2}})$ and $(v^{p_{1}}, v^{p_{2}})$, and this line segment lies in $\Om_{\textup{I}}$.
\end{lemma}
This lemma is obvious from the picture.
\begin{lemma}\label{Bisone}
For every $x$, such that $x\in \Om_{\textup{I}}$, we have $\mathcal{B}(x)=1$.
\end{lemma}
\begin{proof}[Proof of the Lemma \ref{Bisone}]
Take a point $x\in \Om_{\textup{I}}$ and numbers $u, v$ from the Lemma \ref{om1conv}. Then for some $\mu\in[0,1]$ we have $x_{k}=\mu u^{p_{k}}+(1-\mu)v^{p_{k}}$. For this $\mu$ take
$$
w(t)=\begin{cases} u, &t\in [0, \mu) \\ v, &t\in [\mu, 1] \end{cases}.
$$
By the Lemma \ref{kusochnaya}, $w\in A_{p_{1}, p_{2}}^{Q}$. Further,
$$
\ave{w^{p_{k}}}=\mu u^{p_{k}}+(1-\mu)v^{p_{k}}=x_{k}.
$$
We took $u,v\geqslant 1$, thus $|\{ w(t)\geqslant 1 \}|=1$. Since $\ave{w^{p_{k}}}=x_{k}$, and $w\in A_{p_{1}, p_{2}}^{Q}$, we get $\mathcal{B}(x)\geqslant  |\{ w(t)\geqslant 1 \}|=1$. On the other hand, by definition, $\mathcal{B}(x)\leqslant 1$. Therefore, $\mathcal{B}(x)=1$.
\end{proof}
\subsection{Domain $\Omega_{\textup{III}}$}
In this section we find $B$ in $\Om_{\textup{III}}$. As it was said in the subsection \ref{MongeAmpere}, we need to find lines on which $B$ is linear. These lines will simply be lines which connects the point $(1,1)$ with points on $\Gamma$. Our setting is the following: we fix a point $(v^{p_{1}}, v^{p_{2}})\in \Gamma$ and take a line with an equation
$$
v^{p_{2}}(1-x_{1})-v^{p_{1}}(1-x_{2})=x_{2}-x_{1}.
$$
Obviously, this line contains $(1,1)$ and $(v^{p_{1}}, v^{p_{2}})$. We assume that $B$ is linear on our line. Using that $B(1,1)=1$ and $B(v^{p_{1}}, v^{p_{2}})=0$, we get
$$
B(x)=\frac{x_{1}-v^{p_{1}}}{1-v^{p_{1}}}=\frac{x_{2}-v^{p_{2}}}{1-v^{p_{2}}}.
$$
\subsection{Domain $\Omega_{\textup{II}}$}
To find $B$ in $\Om_{\textup{II}}$, we use the following simple observation. $\Om_{\textup{II}}$ has three parts of boundary: parts of $\ell_{\pm}$ and part of $\Gamma_{Q}$. Since our candidate for $\mathcal{B}$ is linear on mentioned parts of $\ell_{\pm}$, it is natural to assume $B$ to be linear in the whole domain $\Om_{\textup{II}}$. Restriction of $B$ to these parts are also linear. It gives us a hope that we can find a fully linear $B$, namely, $B(x)=ax_{1}+bx_{2}+c$. How do we find $a,b,c$? We want $B$ to be continuous on $\ell_{\pm}$ and, therefore, we want $B(1,1)=1$, $B(v_{-}^{p_{1}}, v_{-}^{p_{2}})=0$, $B(v_{+}^{p_{1}}, v_{+}^{p_{2}})=1$. It gives us three equations:
$$ \left\{
  \begin{aligned}
  &a+b+c = 1\\
  &av_-^{p_1}+bv_-^{p_2}+c = 0 \\
  &av_+^{p_1}+bv_+^{p_2}+c = 1
  \end{aligned}
\right. .
$$
Solving this linear system (and using that $v_{+}=\frac{1}{\;v_{-}}$) one gets
$$
\left\{
  \begin{aligned}
  &a=\frac{v_{-}^{p_{1}}}{(1-v_{-}^{p_{1}})(v_{-}^{p_{1}}-v_{-}^{p_{2}})}\\
  &b=\frac{v_{-}^{p_{2}}}{(v_{-}^{p_{2}}-1)(v_{-}^{p_{1}}-v_{-}^{p_{2}})} \\
  &c =1-\frac{1}{(v_{-}^{p_{1}}-1)(v_{-}^{p_{2}}-1)}
  \end{aligned}
\right. .
$$
\subsection{Domain $\Omega_{\textup{IV}}$}
Now we shall find $B$ in $\Om_{\textup{IV}}$. We guess that if $x\in \Om_{\textup{IV}}$ then $B$ is linear on the tangent from $x$ to $\Gamma_{Q}$, which corresponds to $\gamma_{+}$. We remind what it means. For every point $x$ there is a unique point $(v^{p_1}, v^{p_2})\in\Gamma_1$, such that $x$ lies on the line, going from this point, tangent to $\Gamma_Q$ and having $\gamma_+$ in its slope. Namely, the equation of this tangent is
$$
x_{2}=Q^{-p_{2}}\frac{p_{2}}{p_{1}}\gamma_{+}^{p_{2}-p_{1}}v^{p_{2}-p_{1}}(x_{1}-v^{p_{1}})+v^{p_{2}}.
$$
We know that on this line $t_0, t_1, t_2$ are supposed to be constants. It means that they can depend only on $v$.
Therefore, we divide the equation
$$
dt_{0}+x_{1}dt_{1}+x_{2}dt_{2}=0
$$
over $dv$ and get:
\begin{equation}\label{diffff}
t_{0}^{\prime}(v)+x_{1}t_{1}^{\prime}(v)+x_{2}t_{2}^{\prime}(v)=0,
\end{equation}
when
$$
x_{2}=Q^{-p_{2}}\frac{p_{2}}{p_{1}}\gamma_{+}^{p_{2}-p_{1}}v^{p_{2}-p_{1}}(x_{1}-v^{p_{1}})+v^{p_{2}}.
$$
Now we substitute $x_{2}$ from this equation into \eqref{diffff} and use that \eqref{diffff} is true for infinitely many values of $x_{1}$. Therefore, the coefficient of $x_{1}$ must be equal to zero, which yields
\eq[diff]{
t_{1}^{\prime}+Q^{-p_{2}}\frac{p_{2}}{p_{1}}\gamma_{+}^{p_{2}-p_{1}}v^{p_{2}-p_{1}}t_{2}^{\prime}=0.
}
Since $B(v^{p_{1}}, v^{p_{2}})=0$, we get
$$
t_{0}+v^{p_{1}}t_{1}+v^{p_{2}}t_{2}=0, \; \; dt_{0}+v^{p_{1}}dt_{1}+v^{p_{2}}dt_{2}=0,
$$
so
$$
dt_{0}+v^{p_{1}}dt_{1}+p_{1}v^{p_{1}-1}t_{1}dv+v^{p_{2}}dt_{2}+p_{2}v^{p_{2}-1}t_{2}dv=0
$$
or
\eq[tratata]{
p_{1}t_{1}v^{p_{1}-1}+p_{2}t_{2}v^{p_{2}-1}=0,
}
thus
$$
t_{2}^{\prime}=-\frac{p_{1}}{p_{2}}v^{p_{1}-p_{2}-1}(t_{1}^{\prime}v+(p_{1}-p_{2})t_{1}).
$$
Combining the last equation with \eqref{diff} we obtain
$$
t_{1}^{\prime}(1-Q^{-p_{2}}\gamma_{+}^{p_{2}-p_{1}})=\frac{p_{1}-p_{2}}{v}t_{1}Q^{-p_{2}}\gamma_{+}^{p_{2}-p_{1}},
$$
so
$$
t_{1}=Cv^{\frac{(p_{1}-p_{2})Q^{-p_{2}}\gamma_{+}^{p_{2}-p_{1}}}{1-Q^{-p_{2}}\gamma_{+}^{p_{2}-p_{1}}}}=Cv^{\frac{p_1}{\gamma_+^{p_1}-1}.}
$$
From \eqref{tratata} we get
$$
t_{2}=-\frac{p_{1}}{p_{2}}Cv^{{\frac{p_{1}-p_{2}}{1-Q^{-p_{2}}\gamma_{+}^{p_{2}-p_{1}}}}},
$$
and from $t_{0}+v^{p_{1}}t_{1}+v^{p_{2}}t_{2}=0$ we get
$$
t_{0}=\left(\frac{p_{1}}{p_{2}}-1\right)Cv^{{\frac{p_{1}-p_{2}Q^{-p_{2}}\gamma_{+}^{p_{2}-p_{1}}}{1-Q^{-p_{2}}\gamma_{+}^{p_{2}-p_{1}}}}}=\frac{p_1-p_2}{p_1}\cdot C\cdot v^{\frac{p_1 \gamma_+^{p_1}}{\gamma_+^{p_1}-1}}.
$$
We shall find $C$ such that $B$ is continuous in $\Om_{\textup{III}}\cap\Om_{\textup{IV}}$. As before, $A=Q^{-p_{2}}\gamma_{+}^{p_{2}-p_{1}}$. On the line
$$
x_{2}=A\frac{p_{2}}{p_{1}}v_{-}^{p_{2}-p_{1}}x_{1}+v_{-}^{p_{2}}\left(1-\frac{p_{2}}{p_{1}}A\right)
$$
we have
\begin{align*}
B(x_{1}, x_{2})&= C\Bigl[\frac{p_{1}-p_{2}}{p_{2}}v_{-}^{\frac{p_{1}-p_{2}A}{1-A}}+x_{1}v_{-}^{\frac{(p_{1}-p_{2})A}{1-A}}-A\frac{p_{2}}{p_{1}}v_{-}^{p_{2}-p_{1}}x_{1}\frac{p_{1}}{p_{2}}v_{-}^{\frac{p_{1}-p_{2}}{1-A}}-\\
 & \qquad \qquad \qquad \qquad\qquad\qquad\qquad-\frac{p_{1}-p_{2}A}{p_{1}}v_{-}^{p_{2}}\frac{p_{1}}{p_{2}}v_{-}^{\frac{p_{1}-p_{2}}{1-A}}\Bigr]=\\
&=C(1-A)\left[x_{1}v_{-}^{\frac{(p_{1}-p_{2})A}{1-A}}-v_{-}^{\frac{p_{1}-p_{2}A}{1-A}}\right].
\end{align*}
But on this line
$$
B(x_{1}, x_{2})=\frac{x_{1}-v_{-}^{p_{1}}}{1-v_{-}^{p_{1}}},
$$
so
$$
C=\frac{1}{1-A}\frac{v_{-}^{-\frac{(p_{1}-p_{2})A}{1-A}}}{1-v_{-}^{p_{1}}}.
$$
\subsection{Answer for $B$}
Now we would like to state the answer for $B$.
\begin{equation}\label{answerforb}
B(x)=\begin{cases} 1, &x\in\Om_{\textup{I}} \\
                   \frac{v_{-}^{p_{1}}}{(1-v_{-}^{p_{1}})(v_{-}^{p_{1}}-v_{-}^{p_{2}})}x_{1}+\frac{v_{-}^{p_{2}}}{(v_{-}^{p_{2}}-1)(v_{-}^{p_{1}}-v_{-}^{p_{2}})}x_{2}+ 1-\frac{1}{(v_{-}^{p_{1}}-1)(v_{-}^{p_{2}}-1)}, &x\in \Om_{\textup{II}}\\ \\
                   \frac{x_{1}-v^{p_{1}}}{1-v^{p_{1}}}, &x\in \Om_{\textup{III}}\\ \\
                   \frac{1}{1-A}\frac{v_{-}^{-\frac{(p_{1}-p_{2})A}{1-A}}}{1-v_{-}^{p_{1}}}v^{\frac{p_{1}-p_{2}}{1-A}} \left( \frac{p_{1}-p_{2}}{p_{2}}v^{p_{2}}+x_{1}v^{p_{2}-p_{1}}-\frac{p_{1}}{p_{2}}x_{2}\right), &x\in \Om_{\textup{IV}} \end{cases},
\end{equation}
where $v$ is defined as a solution of equation
$$
\begin{cases} v^{p_{2}}(1-x_{1})-v^{p_{1}}(1-x_{2})=x_{2}-x_{1}, &x\in \Om_{\textup{III}} \\
              x_{2}=Q^{-p_{2}}\frac{p_{2}}{p_{1}}\gamma_{+}^{p_{2}-p_{1}}v^{p_{2}-p_{1}}(x_{1}-v^{p_{1}})+v^{p_{2}}, &x\in \Om_{\textup{IV}}
\end{cases},
$$
and has the property that $\sig(x_{1}-v^{p_{1}})=\sig(p_{1})$.
\section{The estimate from above: $B\geqslant \mathcal{B}$}\label{bolshe}
In the proof of the Theorem \ref{th1} we will use the following lemma. However, we postpone its prove to the Section \ref{profconc}.
\begin{lemma}\label{lemma-1-1}
$B$ is locally concave.
\end{lemma}

We would like to prove that $B(x)\geqslant \mathcal{B}(x)$. To do it, we enlarge our $\Omega$ and consider another $B$. Precisely, take a number $Q_1$, $Q_1>Q$, and a domain $\Om_{Q_1}=\{x=(x_{1}, x_{2}): 1\leqslant x_{1}^{\frac{1}{p_{1}}}x_{2}^{-\frac{1}{p_{2}}}\leqslant Q_1 \}$. Now for this domain we build a function $B_{Q_1}(x)$ in the same way we built our $B$ in $\Om$. In fact, we use formulas \eqref{answerforb}, but we change $Q$ by $Q_1$. Note that $\Om\subset \Om_{Q_1}$ and $A_{p_{1}, p_{2}}^{Q}\subset A_{p_{1}, p_{2}}^{Q_1}$. We need the following lemma.
\begin{lemma}\label{lemma1-2}
Fix $Q_1>Q>1$. Then for every $w\in A_{p_{1}, p_{2}}^{Q}$ there are two intervals $I^{+}$ and $I^{-}$ such that $I=I^{-}\cup I^{+}$ and if $x^{\pm}=\left( \ave{w^{p_{1}}}_{I_{\pm}}, \ave{w^{p_{2}}}_{I_{\pm}}\right)$ then $[x^{-}, x^{+}]\subset \Om_{Q_1}$.
Also the parameters $\alpha^{\pm}$ can be taken separated from $0$ and $1$ uniformly with respect to $w$.
\end{lemma}
This lemma was proved in ~\cite{Va}. Using it we prove the following theorem.
\begin{theorem}\label{th1}
For every point $x\in \Omega$ and for every $Q_1>Q$ we have $B_{Q_1}(x)\geqslant \mathcal{B}(x)$.
\end{theorem}
\begin{collorary}
For every point $x\in \Omega$ we have $B(x)\geqslant \mathcal{B}(x)$.
\end{collorary}
\begin{proof}[Proof of Corollary]
It is obvious since $B_{Q}(x)$ is continuous with respect to $Q$.
\end{proof}
\begin{proof}[Proof of the Theorem]
We want to prove that for any function $w\in A_{p_{1}, p_{2}}^{Q}(I)$ and $x=(\ave{w^{p_{1}}}, \ave{w^{p_{2}}})$ it is true that
\begin{equation}\label{tirlimbombom}
B_{Q_1}(x)\geqslant |\{w\geqslant 1\}|.
\end{equation}
Then, passing to the supremum in the right-hand side, we get what we need.
Assume $w\in A_{p_{1}, p_{2}}^{Q}$.
We take a splitting of our interval $I$ by the rule from the Lemma \ref{lemma1-2}. By $D_{n}$ we denote the set of intervals of $n$-th generation. For example, $D_{0}=\{I\}$ and $D_{1}=\{ I^{-}, I^{+}\}$. For every interval $J\in D_{n}$ we denote
$$
x^{J}=(\av{w^{p_{1}}}{J}, \av{w^{p_{2}}}{J}).
$$

Since $B_{Q_1}(x)$ is locally concave, we get
\begin{equation}\label{chitoguritto}
B_{Q_1}(x)\geqslant |I^{-}|B_{Q_1}(x^{I^-})+|I^{+}|B_{Q_1}(x^{I^+})\geqslant \sli_{J\in D_{n}}|J|B_{Q_1}(x^{J})=\ili_{0}^{1} B_{Q_1}(x^{n}(t))dt,
\end{equation}
where
$x^{n}(t)$ is a step-function, defined in the following way: take $J\in D_{n}$ and denote $x^{n}(t)=x^{J}, \; t\in J$.

Since we assume that $w^{p_i}\in L_{1, loc}$, we get
$$
x^{n}(t)\to (w^{p_{1}}(t), w^{p_{2}}(t)) \; a.e.
$$
Moreover, in the Section \ref{menshe} it will be proved that for every $x\in \Omega$ there exists a function $w\in A_{p_{1}, p_{2}}^{Q}$ such that $B(x)=|\{w\geqslant 1\}|$. The same can be shown for $Q_1$ instead of $Q$, so we get $B_{Q_1}(x)\leqslant 1$. Therefore, by the Lebesgue Dominated Convergence Theorem, we can pass to the limit in \eqref{chitoguritto}. Then we get
$$
B_{Q_1}(x)\geqslant \ili_{0}^{1}B_{Q_1}(w^{p_{1}}(t), w^{p_{2}}(t))dt.
$$
But for every $t$ we have $(w^{p_{1}}(t), w^{p_{2}}(t))\in \Gamma$, where we know $B_{Q_1}$ by the Lemma ~\ref{boundlambda}. Therefore,
$$
B_{Q_1}(x)\geqslant |\{t: w(t)\geqslant 1\}|,
$$
which is what we need.
\end{proof}
\section{Proof of concavity}\label{profconc}
In this section we prove the following lemma.
\begin{lemma}
The following inequality holds in the sense of distributions:
$$
\frac{d^{2}B}{dx^{2}}\leqslant 0.
$$
\end{lemma}

We break the proof of this lemma into parts. According to the paragraph \ref{special}, first we check that in interiors of $\Om_{\textup{I}}$---$\Om_{\textup{IV}}$ the Hessian of $B$ is nonpositive.

Then we study jumps of $B_{x_{1}}$ and $B_{x_{2}}$ over $\ell_{\pm}$.

We warn the reader that this section is rather technical.
\subsection{Domains $\Om_{\textup{I}}$ and $\Om_{\textup{II}}$}
Here $B$ is fully linear and, therefore, $\frac{d^{2}B}{dx^{2}}=0$.
\subsection{Domain $\Om_{\textup{III}}$}
As we know, here
$$
B(x)=\frac{x_{2}-v^{p_{2}}}{1-v^{p_{2}}}=\frac{x_{2}-1}{1-v^{p_{2}}}+1.
$$
Recall that $v^{p_{2}}(1-x_{1})-v^{p_{1}}(1-x_{2})=x_{2}-x_{1}$ so $$
(p_{2}v^{p_{2}-1}(1-x_{1})-p_{1}v^{p_{1}-1}(1-x_{2}))v^{\prime}_{x_{1}}-v^{p_{2}}=-1,$$ or
$$
v^{\prime}_{x_{1}}=v\frac{v^{p_{2}}-1}{\Upsilon},
$$
where
$$
\Upsilon=\Upsilon(v)=p_{2}v^{p_{2}}(1-x_{1})-p_{1}v^{p_{1}}(1-x_{2}).
$$
Put
$$
f(v)=\frac{v^{p_{2}}}{v^{p_{2}}-1}=1+\frac{1}{v^{p_{2}}-1}.
$$
Once again,
$$
B(x)=\frac{x_{2}-v^{p_{2}}}{1-v^{p_{2}}}=\frac{x_{2}-1}{1-v^{p_{2}}}+1,
$$
so
\begin{multline}\label{derivOm3}
B^{\prime}_{x_{1}}=(x_{2}-1)\cdot\frac{p_{2}v^{p_{2}-1}}{(1-v^{p_{2}})^{2}}v^{\prime}_{x_{1}}=p_{2}(x_{2}-1)\cdot\frac{v^{p_{2}-1}}{(1-v^{p_{2}})^{2}}\cdot\frac{v(v^{p_{2}}-1)}{\Upsilon}=\\=p_{2}(x_{2}-1)\cdot\frac{v^{p_{2}}}{v^{p_{2}}-1}\cdot\frac{1}{\Upsilon}=p_{2}(x_{2}-1)\frac{f(v)}{\Upsilon}.
\end{multline}
Observe that
$$
f^{\prime}_{x_{1}}(v)=-p_{2}\cdot\frac{v^{p_{2}-1}}{(v^{p_{2}}-1)^{2}}\cdot v^{\prime}_{x_{1}}=-p_{2}\cdot\frac{f(v)}{\Upsilon},
$$
therefore
\begin{align*}
&B^{\prime\prime}_{x_{1}x_{1}}=\\&=p_{2}(x_{2}-1)\left[-p_{2}\frac{f(v)}{\Upsilon^{2}}-\frac{(p_{2}^{2}v^{p_{2}-1}(1-x_{1})-p_{1}^{2}v^{p_{1}-1}(1-x_{2}))v^{\prime}_{x_{1}}-p_{2}v^{p_{2}}}{\Upsilon^{2}}f(v)\right]\\
&=\frac{p_{2}(1-x_{2})f(v)}{\Upsilon^{2}}\left[p_{2}-p_{2}v^{p_{2}}+(p_{2}^{2}v^{p_{2}-1}(1-x_{1})-p_{1}^{2}v^{p_{1}-1}(1-x_{2}))\frac{v(v^{p_{2}}-1)}{\Upsilon}\right]\\
&=-\frac{p_{2}(x_{2}-1)f(v)}{\Upsilon^{2}}(v^{p_{2}}-1)\left[\frac{p_{2}^{2}v^{p_{2}}(1-x_{1})-p_{1}^{2}v^{p_{1}}(1-x_{2})}{p_{2}v^{p_{2}}(1-x_{1})-p_{1}v^{p_{1}}(1-x_{2})}-p_{2}\right]\\
&=-\frac{p_{2}(x_{2}-1)f(v)(v^{p_{2}}-1)}{\Upsilon^{2}}\cdot\frac{p_{2}p_{1}v^{p_{1}}(1-x_{2})-p_{1}^{2}v^{p_{1}}(1-x_{2})}{\Upsilon}\\
&=-\frac{p_{1}p_{2}(x_{2}-1)f(v)(v^{p_{2}}-1)v^{p_{1}}(1-x_{2})(p_{2}-p_{1})}{\Upsilon^{3}}\\
&=\frac{(x_{2}-1)^{2}(p_{2}-p_{1})v^{p_{1}+p_{2}}}{\Upsilon^{2}}\cdot\frac{p_{1}p_{2}}{\Upsilon}.
\end{align*}
Now we calculate $B^{\prime\prime}_{x_{2}x_{2}}$. We use that
$$
B(x)=\frac{x_{1}-v^{p_{1}}}{1-v^{p_{1}}}=\frac{x_{1}-1}{1-v^{p_{1}}}+1.
$$
By straight-forward calculations we get
$$
B^{\prime\prime}_{x_{2}x_{2}}=\frac{(x_{1}-1)^{2}(p_{2}-p_{1})v^{p_{1}+p_{2}}}{\Upsilon^{2}}\cdot\frac{p_{1}p_{2}}{\Upsilon}.
$$
Using that $\det \frac{d^{2}B}{dx^{2}}=0$, we immediately get
$$
B^{\prime\prime}_{x_{1}x_{2}}=\pm\frac{(1-x_{1})(1-x_{2})v^{p_{1}+p_{2}}(p_{2}-p_{1})}{\Upsilon^{2}}\cdot\frac{p_{1}p_{2}}{\Upsilon},
$$
and we do not care if there is plus or minus.

Finally,
$$
\sli_{i,j}B^{\prime\prime}_{x_{i}x_{j}}\Delta_{i}\Delta_{j}=\frac{v^{p_{1}+p_{2}}(p_{2}-p_{1})}{\Upsilon^{2}}\cdot\frac{p_{1}p_{2}}{\Upsilon}\left((1-x_{1})\Delta_{1}\pm (1-x_{2})\Delta_{2}\right)^{2}.
$$
Recall that
$$
\Upsilon=v\cdot \frac{v^{p_{2}}-1}{v^{\prime}_{x_{1}}},
$$
so
\begin{multline}
\sli_{i,j}B^{\prime\prime}_{x_{i}x_{j}}\Delta_{i}\Delta_{j}=\frac{v^{p_{1}+p_{2}}(p_{2}-p_{1})}{\Upsilon^{2}}\cdot\frac{1}{v}\cdot\frac{p_{1}p_{2}v^{\prime}_{x_{1}}}{v^{p_{2}}-1}\left((1-x_{1})\Delta_{1}\pm (1-x_{2})\Delta_{2}\right)^{2}=\\
=\frac{v^{p_{1}+p_{2}}(p_{2}-p_{1})}{\Upsilon^{2}}\cdot\frac{1}{v}\cdot\frac{p_{2}}{v^{p_{2}}-1}\cdot p_{1}v^{\prime}_{x_{1}}\left((1-x_{1})\Delta_{1}\pm (1-x_{2})\Delta_{2}\right)^{2}.
\end{multline}
Observe that $\sig(v^{p_{2}}-1)=-\sig(p_{2})$, $\sig{v^{\prime}_{x_{1}}}=-\sig{p_{1}}$ and $p_{2}-p_{1}<0$. It gives that
$$
\sli_{i,j}B^{\prime\prime}_{x_{i}x_{j}}\Delta_{i}\Delta_{j}\leqslant 0.
$$
\subsection{Domain $\Om_{\textup{IV}}$}
We know that $$
B^{\prime}_{x_{1}}=t_{1}=\frac{1}{1-A}\cdot\frac{v_{-}^{-\alpha}}{1-v_{-}^{p_{1}}}v^{\frac{(p_{1}-p_{2})A}{1-A}}.$$
We do not need to write the full expression for $\alpha$, since it does not matter for the sign of anything. Moreover, put $V_{-}=\frac{v_{-}^{-\alpha}}{1-v_{-}^{p_{1}}}$. Then we get
$$
B^{\prime\prime}_{x_{1}x_{1}}=\frac{(p_{1}-p_{2})A}{(1-A)^{2}}V_{-}v^{\frac{(p_{1}-p_{2})A}{1-A}-1}v^{\prime}_{x_{1}}.
$$
Similarly,
$$
B^{\prime}_{x_{2}}=t_{2}=-\frac{p_{1}}{p_{2}}\cdot\frac{1}{1-A}V_{-}v^{\frac{p_{1}-p_{2}}{1-A}},
$$
$$
B^{\prime\prime}_{x_{2}x_{2}}=-\frac{p_{1}}{p_{2}}\cdot\frac{p_{1}-p_{2}}{(1-A)^{2}}V_{-}v^{\frac{p_{1}-p_{2}}{1-A}-1}v^{\prime}_{x_{2}},
$$
$$
B^{\prime\prime}_{x_{1}x_{2}}=B^{\prime\prime}_{x_{2}x_{1}}=-\frac{p_{1}}{p_{2}}\cdot\frac{p_{1}-p_{2}}{(1-A)^{2}}V_{-}v^{\frac{p_{1}-p_{2}}{1-A}-1}v^{\prime}_{x_{1}}.
$$
As we know from the Lemma \ref{depend},
$$
v^{\prime}_{x_{1}}=\frac{1}{p_{1}\Pi}\cdot\frac{A}{v^{p_{1}}},
$$
$$
v^{\prime}_{x_{2}}=-\frac{1}{p_{2}\Pi}\cdot\frac{1}{v^{p_{2}}}.
$$
We have
$$
\sig(B^{\prime\prime}_{x_1x_1})=\sig(1-v_-^{p_1})\sig(v^{\prime}_{x_1})=-1.
$$
Similarly $\sig(B^{\prime\prime}_{x_2x_2})=-1$, and since $\det(\frac{d^{2}B}{dx^2})=0$, we get that $\frac{d^{2}B}{dx^2}\leqslant 0$.
\subsection{Boundary}
Now we proceed to ``jumps'' of first derivatives of $B$. We remind that in $\Om_{\textup{II}}$ we have
$$
B(x)=ax_{1}+bx_{2}+c,
$$
where
$$
b=\frac{v_{-}^{p_{2}}}{(v_{-}^{p_{2}}-1)(v_{-}^{p_{1}}-v_{-}^{p_{2}})}.
$$
We also remind the following notation. As before, for a smooth compactly supported test function $\vf$, we introduce two functionals
$$
\Phi_{\pm}(\vf)=\ili \vf(x_{1}, Q^{-p_{2}}\frac{p_{2}}{p_{1}}\gamma_{\pm}^{p_{2}-p_{1}}(x_{1}-1)+1)dx_{1}.
$$

\subsubsection{Boundary $\Om_{\textup{I}}\cap \Om_{\textup{II}}$}
Observe that if $p_{2}>0$ then
$$
B^{\prime}_{x_{2}}=\begin{cases} 0, &x_{2}>Q^{-p_{2}}\frac{p_{2}}{p_{1}}\gamma_{+}^{p_{2}-p_{1}}(x_{1}-1)+1 \\
b, &x_{2}<Q^{-p_{2}}\frac{p_{2}}{p_{1}}\gamma_{+}^{p_{2}-p_{1}}(x_{1}-1)+1\end{cases}.
$$
and if $p_{2}<0$ then
$$
B^{\prime}_{x_{2}}=\begin{cases} b, &x_{2}>Q^{-p_{2}}\frac{p_{2}}{p_{1}}\gamma_{+}^{p_{2}-p_{1}}(x_{1}-1)+1 \\
0, &x_{2}<Q^{-p_{2}}\frac{p_{2}}{p_{1}}\gamma_{+}^{p_{2}-p_{1}}(x_{1}-1)+1\end{cases}.
$$

Therefore, in the sense of distributions,
$$
B^{\prime\prime}_{x_{2}x_{2}}=-\sig(p_{2})b\Phi_{+}.
$$
Notice that $\Phi_+$ is a non-negative functional, and therefore sign of $B^{\prime\prime}_{x_2x_2}$ is defined by the sign of $-\sig(p_2)b$.
Since
$$
b=\frac{v_{-}^{p_{2}}}{(v_{-}^{p_{2}}-1)(v_{-}^{p_{1}}-v_{-}^{p_{2}})},
$$
and
\begin{align}\label{tralyalya1}
&\sig(v_{-}^{p_{2}}-1)=-\sig(p_{2}),\\ \label{tralyalya2}
&v_{-}^{p_{1}}-v_{-}^{p_{2}}=v_{-}^{p_{1}}(1-v_{-}^{p_{2}-p_{1}})<0,
\end{align}
we get
$$
\sig(b)=\sig(p_{2}),
$$
so
$$
B^{\prime\prime}_{x_{2}x_{2}}\leqslant 0.
$$
\subsubsection{Boundary $\Om_{\textup{II}}\cap \Om_{\textup{III}}$}
Our plan is the following. First we count $B^{\prime}_{x_{k}}$ on the line $\frac{x_{2}-1}{v_{-}^{p_{2}}-1}=\frac{x_{1}-1}{v_{-}^{p_{1}}-1}$, i.e. on $\ell_{-}$. Then we proceed to jumps.

Let us evaluate $B_{x_{k}}$ on $\ell_{-}$.
In $\Om_{\textup{III}}$, from \eqref{derivOm3} we have
$$
B_{x_{1}}=p_{2}(x_{2}-1)\frac{v^{p_{2}}}{v^{p_{2}}-1}\cdot\frac{1}{p_{2}v^{p_{2}}(1-x_{1})-p_{1}v^{p_{1}}(1-x_{2})}.
$$
On $\ell_{-}$ we have $v=v_{-}$, so, using the equation of $\ell_{-}$ and canceling $x_{2}-1$, we get:
\begin{multline*}
B_{x_{1}}(\ell_{-})=p_{2}(x_{2}-1)\frac{v_{-}^{p_{2}}}{v_{-}^{p_{2}}-1}\cdot\frac{1}{p_{2}v_{-}^{p_{2}}(1-x_{1})-p_{1}v_{-}^{p_{1}}(1-x_{2})}=\\=p_{2}v_{-}^{p_{2}}\frac{1}{p_{2}v_{-}^{p_{2}}(1-v_{-}^{p_{1}})-p_{1}v_{-}^{p_{1}}(1-v_{-}^{p_{2}})}.
\end{multline*}
Now we use that $(v_{-}^{p_{1}}, v_{-}^{p_{2}})\in\ell_{-}$, i.e.,
$$
v_{-}^{p_{2}}=Q^{-p_{2}}\frac{p_{2}}{p_{1}}\gamma_{-}^{p_{2}-p_{1}}(v_{-}^{p_{1}}-1)+1.
$$
It gives
\begin{multline}
B_{x_{1}}(\ell_{-})=\frac{p_{2}v_{-}^{p_{2}}}{p_{2}v_{-}^{p_{2}}(1-v_{-}^{p_{1}})-p_{2}Q^{-p_{2}}\gamma_{-}^{p_{2}-p_{1}}(1-v_{-}^{p_{1}})v_{-}^{p_{1}}}= \\= \frac{v_{-}^{p_{2}}}{1-v_{-}^{p_{1}}}\cdot\frac{1}{v_{-}^{p_{2}}-Q^{-p_{2}}\gamma_{-}^{p_{2}-p_{1}}v_{-}^{p_{1}}}=\frac{1}{1-v_{-}^{p_{1}}}\cdot\frac{1}{1-Q^{-p_{2}}\gamma_{-}^{p_{2}-p_{1}}v_{-}^{p_{1}-p_{2}}}.
\end{multline}
Observe that $v_{-}=\frac{\gamma_{-}}{\gamma_{+}}$, so
$$
B_{x_{1}}(\ell_{-})=\frac{1}{1-v_{-}^{p_{1}}}\cdot\frac{1}{1-Q^{-p_{2}}\gamma_{+}^{p_{2}-p_{1}}}.
$$
Finally, $\sig(1-v_{-}^{p_{1}})=\sig(p_{1})$ and by Lemma \ref{estimate} we have $1-Q^{-p_{2}}\gamma_{+}^{p_{2}-p_{1}}>0$. Therefore,
$$
\sig(B_{x_{1}}(\ell_{-}))=\sig{p_{1}}.
$$

Similarly,
$$
B_{x_{2}}(\ell_{-})=-\frac{p_{1}}{p_{2}}\cdot\frac{1}{1-Q^{-p_{2}}\gamma_{+}^{p_{2}-p_{1}}}\cdot\frac{v_{-}^{p_{1}-p_{2}}}{1-v_{-}^{p_{1}}},
$$
and
$$
\sig(B_{x_{2}}(\ell_{-}))=-\sig{p_{2}}.
$$
As before, we observe that if $p_{2}>0$ then
$$
B_{x_{2}}=\begin{cases} B_{x_{2}}(\ell_{-}), &x_{2}\geqslant Q^{-p_{2}}\frac{p_{2}}{p_{1}}\gamma_{-}^{p_{2}-p_{1}}(x_{1}-1)+1\\
b, &x_{2}\leqslant Q^{-p_{2}}\frac{p_{2}}{p_{1}}\gamma_{-}^{p_{2}-p_{1}}(x_{1}-1)+1
\end{cases}.
$$
and if $p_{2}<0$ then
$$
B_{x_{2}}=\begin{cases} b, &x_{2}\geqslant Q^{-p_{2}}\frac{p_{2}}{p_{1}}\gamma_{-}^{p_{2}-p_{1}}(x_{1}-1)+1\\
 B_{x_{2}}(\ell_{-}), &x_{2}\leqslant Q^{-p_{2}}\frac{p_{2}}{p_{1}}\gamma_{-}^{p_{2}-p_{1}}(x_{1}-1)+1
\end{cases}.
$$
Therefore,
$$
(B_{x_{2}x_{2}},\vf)=\sig(p_{2})(B_{x_{2}}(\ell_{-})-b)\Phi_{-}(\vf)+\ili B_{x_{2}x_{2}}\vf dx.
$$
Moreover,
$$
\sig(B_{x_{2}}(\ell_{-}))=-\sig(p_{2}),
$$
$$
\sig(b)=\sig(p_{2}),
$$
so
$$
\sig(p_{2})(B_{x_{2}}(\ell_{-})-b)\leqslant 0.
$$
It finishes the proof.
\subsubsection{Boundary $\Om_{\textup{III}}\cap \Om_{\textup{IV}}$}
This is the best boundary since here all derivatives of $B$ are continuous. We check it straightforward. We already know values of $B$ when we approach $\ell_{-}$ from $\Om_{\textup{III}}$.

Observe that in $\Om_{\textup{IV}}$
$$
B(x)=\frac{1}{1-A}\frac{v_{-}^{-\frac{(p_{1}-p_{2})A}{1-A}}}{1-v_{-}^{p_{1}}}v^{\frac{p_{1}-p_{2}}{1-A}} \left( \frac{p_{1}-p_{2}}{p_{2}}v^{p_{2}}+x_{1}v^{p_{2}-p_{1}}-\frac{p_{1}}{p_{2}}x_{2}\right).
$$
Also we know that the solution of Monge-Amp\`{e}re equation satisfies the following: $B^{\prime}_{x_{1}}=t_{1}$, $B^{\prime}_{x_{2}}=t_{2}$. Thus,
$$
B^{\prime}_{x_{1}}=t_{1}=\frac{1}{1-A}\frac{v_{-}^{-\frac{(p_{1}-p_{2})A}{1-A}}}{1-v_{-}^{p_{1}}}v^{\frac{(p_{1}-p_{2})A}{1-A}}.
$$
Now we plug $v=v_{-}$. Then we immediately get
$$
t_{1}(v_{-})=\frac{1}{1-A}\frac{1}{1-v_{-}^{p_{1}}}.
$$
Recalling that $A=Q^{-p_{2}}\gamma_{-}^{p_{2}-p_{1}}$ gives us that $B_{x_{1}}$ is continuous on $\ell_{-}$.

Moreover,
$$
t_{2}=-\frac{p_{1}}{p_{2}}\frac{1}{1-A}\frac{v_{-}^{-\frac{(p_{1}-p_{2})A}{1-A}}}{1-v_{-}^{p_{1}}}v^{\frac{p_{1}-p_{2}}{1-A}},
$$
so
$$
t_{2}(v_{-})=-\frac{p_{1}}{p_{2}}\frac{1}{1-A}\frac{v_{-}^{p_{1}-p_{2}}}{1-v_{-}^{p_{1}}}=B_{x_{2}}(\ell_{-}),
$$
which finishes the proof.
\begin{zamech}
Observe that it was sufficient to check continuity only of one derivative.
\end{zamech}
We shall briefly explain that. Note that the gradient of $B$ is a vector $\nabla B=(B^{\prime}_{x_{1}}, B^{\prime}_{x_{2}})$. Assume we have two basis vectors $e_{1}, e_{2}$ in $\R^{2}$. Then to prove that $\nabla B$ is continuous it is sufficient to prove that $\nabla B\cdot e_{i}$ is continuous for $i=1,2$. We choose $e_{1}$ to be a vector, which is parallel to $\ell_{-}$. Since $\nabla B$ is constant on $\ell_{-}$, we obviously have that $\nabla B\cdot e_{1}$ is continuous. Now we simply choose $e_{2}=(1,0)$ or $e_{2}=(0,1)$, which is it.
\section{The estimate from below: $B\leqslant \mathcal{B}$. Constructing test-functions}\label{menshe}

\subsection{Domain $\Om_{\textup{I}}$}
In the Subsection \ref{subsectionomega1} we have already proved that for every point $x\in \Om_{\textup{I}}$ there is a suitable function $w$ such that $|\{t\colon w(t)\geqslant 1\}|=1$.
\subsection{Domain $\Om_{\textup{II}}$}
We proceed with the same idea as in the Subsection \ref{subsectionomega1}. Take $x\in \Om_{\textup{II}}$ and an interval $[x^{-},x^{+}]\subset \Om_{\textup{II}}$ such that $x\in [x^{-}, x^{+}]$ and $x^{\pm}\in \ell_{\pm}$. We write
$$
x=\la x^{+}+(1-\la)x^{-}.
$$
We also can write that
$$
x_{i}^{\pm}=\mu_{\pm}+(1-\mu_{\pm})v_{\pm}^{p_i}.
$$
We know that test-functions for $x^{\pm}$ are the following functions:
$$
w^{-}(t)=\begin{cases} v_{-}, &t\in [0, 1-\mu_{-}] \\
                       1,     &t\in [1-\mu_{-}, 1]
                       \end{cases};
$$
$$
w^{+}(t)=\begin{cases} 1, &t\in [0, \mu_{+}] \\
                       v_{+},     &t\in [\mu_{+}, 1]
                       \end{cases}.
$$
We wrote this $w^{+}$ to emphasize how we want to ``glue'' $w^{\pm}$ together to get a test-function for $x$. We will do the following: $x^{-}$ has weight $1-\la$ and in the definition of $w^{-}$ we have $v_{-}$ with weight $1-\mu_{-}$. Therefore, we define $w(t)=v_{-}$ on an interval of length $(1-\la)(1-\mu_{-})$. Using the same idea we deduce on which intervals $w$ has to be equal to $1$ and to $v_{+}$. Here is the definition:
$$
w(t)=\begin{cases}
           v_{-}, &t\in [0, (1-\la)(1-\mu_{-})]\\
           1,     &t\in ((1-\la)(1-\mu_{-}), 1-\la+\la\mu_{+}]\\
           v_{+}, &t\in (1-\la+\la\mu_{+}, 1]
    \end{cases}.
$$
Note that
\begin{multline}
$$
\ave{w^{p_{i}}}=v_{-}^{p_{i}}(1-\la)(1-\mu_{-})+1-\la+\la\mu_{+}-(1-\la)+(1-\la)\mu_{-}+v_{+}^{p_{i}}(\la-\la\mu_{+})=\\
=(1-\la)(\mu_{-}+(1-\mu_{-})v_{-}^{p_{i}})+\la(\mu_{+}+(1-\mu_{+})v_{+}^{p_{i}})=x_{i}.
$$
\end{multline}
Now our goal is to prove that for every $[\alpha, \beta]\subset [0,1]$ we have
$$
\av{w^{p_{1}}}{[\alpha, \beta]}^{\frac{1}{p_{1}}}\av{w^{p_{2}}}{[\alpha, \beta]}^{-\frac{1}{p_{2}}}\leqslant Q
$$
and that $B(x)=|\{t\colon w(t)\geqslant 1\}|$.
It follows from the next two lemmata.
\begin{lemma}\label{tilitili}
The point $(\av{w^{p_{1}}}{[\alpha, \beta]}, \av{w^{p_{2}}}{[\alpha, \beta]})$ is a convex combination of $x^{-}$, $x^{+}$ and $(1,1)$ or lies on $\ell_{\pm}$.
\end{lemma}
\begin{lemma}\label{tralivali}
$B(x)=|\{t\colon w(t)\geqslant 1\}|$.
\end{lemma}
\begin{proof}[Proof of the Lemma \ref{tilitili}]
Proof of this lemma is similar to the proof of the Lemma \ref{kusochnaya}.
\begin{center}
\includegraphics{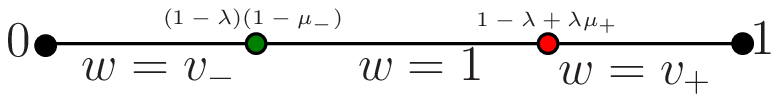}
\end{center}

It is easy to see that the only interesting case is $\alpha<(1-\la)(1-\mu_{-})$ and $\beta>1-\la+\la\mu_{+}$. In other cases we have a convex combination of $v_{-}$ and $(1,1)$ or a convex combination of $v_{+}$ and $(1,1)$,
which follows exactly from the Lemma \ref{kusochnaya}.
If $\alpha<(1-\la)(1-\mu_{-})$ and $\beta>1-\la+\la\mu_{+}$ then
\begin{align*}
\av{w^{p_{i}}}{[\alpha, \beta]}&=\frac{1}{\beta-\alpha}\Bigl[v_{-}^{p_{i}}((1-\la)(1-\mu_{-})-\alpha)+(\la\mu_{+}+(1-\la)\mu_{-})+\\
&\qquad\qquad\qquad\qquad\qquad\qquad+v_{+}^{p_{i}}(\beta-1+\la-\la\mu_{+}) \Bigr]=\\
&= \frac{1}{\beta-\alpha}\Bigl[\frac{((1-\la)(1-\mu_{-})-\alpha)}{1-\mu_{-}}\cdot x_{i}^{-}+ \frac{\beta-1+\la-\la\mu_{+}}{1-\mu_{+}}\cdot x_{i}^{+}+\\
&\qquad\qquad\qquad\qquad\qquad\qquad+\left( \alpha \frac{\mu_{-}}{1-\mu_{-}} + (1-\beta) \frac{ \mu_{+}}{1-\mu_{+}}\right)\cdot 1\Bigr].
\end{align*}
Note that the sum of all coefficients is equal to one, so we have a convex combination of $x^{\pm}$ and $(1,1)$.
\end{proof}
\begin{proof}[Proof of the Lemma \ref{tralivali}]
Since $B$ is linear in $clos(\Om_{\textup{II}})$ we get
\begin{multline}
B(x)=\la B(x^{+})+(1-\la)B(x^{-})=\\=\la (1-\mu_{+})B(v_{+}^{p_{1}}, v_{+}^{p_{2}})+\la\mu_{+}B(1,1)+(1-\la)\mu_{-}B(1,1)+(1-\la)(1-\mu_{-})B(v_{-}^{p_{1}}, v_{-}^{p_{2}})=\\=\la (1-\mu_{+})+\la\mu_{+}+(1-\la)\mu_{-}=\la+(1-\la)\mu_{-};
\end{multline}
$$
|\{w\geqslant 1\}|=1-(1-\la)(1-\mu_{-})=\la+\mu_{-}-\la\mu_{-},
$$
which finishes the proof.
\end{proof}
\subsection{Domain $\Om_{\textup{III}}$}
Take a point $x\in \Om_{\textup{III}}$ and connect it with the point $(1,1)$. Take a number $v<1$, such that the point $(v^{p_{1}}, v^{p_{2}})$ lies on the line segment, which connects $(1,1)$ and $x$.
Then there is a number $\mu\in [0,1]$ such that $x_{1}=\mu+(1-\mu)v^{p_{1}}$, and $x_{2}=\mu + (1-\mu) v^{p_{2}}$.

Denote
$$
w(t)=\begin{cases} 1, &t\in [0, \mu] \\
                   v, &t\in [\mu, 1]
                   \end{cases}.
$$
By the Lemma \ref{kusochnaya}, $w\in A_{p_{1}, p_{2}}^{Q}$ and $\ave{w^{p_{k}}}=x_{k}$.

Moreover, since $v<1$, we get
$$
\mathcal{B}(x)\geqslant|\{t\colon w(t)\geqslant 1\}|=\mu=\frac{x_{1}-v^{p_{1}}}{1-v^{p_{1}}}=B(x).
$$
\subsection{Domain $\Om_{\textup{IV}}$}
Our plan is the following. First we consider $x\in \Gamma_{Q}$. After we build a suitable function $w$ for every such $x$, it will be easy to construct a function for every $x\in \Om_{\textup{IV}}$.
\subsubsection{The case $x\in\Gamma_{Q}$}
Take $x\in \Gamma_{Q}\cap \Om_{\textup{IV}}$. As usual, in $\Om_{\textup{IV}}$ we have
$$
x_{1}=\gamma_{+}^{p_{1}}v^{p_{1}},
$$
$$
x_{2}=Q^{-p_{2}}\gamma_{+}^{p_{2}}v^{p_{2}}.
$$
To introduce $w$ we need some notation. First, choose $\nu$ such that
$$
\frac{1}{1-\nu p_{1}}=\gamma_{+}^{p_{1}}.
$$
Take now
$$
a=\left(\frac{v}{v_{-}}\right)^{\frac{1}{\nu}}.
$$
We also remind that
$$
v_{-}=\frac{\gamma_{-}}{\gamma_{+}}.
$$
Now denote
$$
w(t)=\begin{cases}1, &t\in [0, \frac{\gamma_{-}^{p_{1}}-v_{-}^{p_{1}}}{1-v_{-}^{p_{1}}}\cdot a], \\
                  v_{-}, &t\in [\frac{\gamma_{-}^{p_{1}}-v_{-}^{p_{1}}}{1-v_{-}^{p_{1}}}\cdot a, a],\\
                  v_{-}\cdot \left(\frac{a}{t}\right)^{\nu}, &t\in[a,1].
\end{cases}
$$
As usual, here is what we want.
\begin{lemma}\label{koordinatiw}
$\ave{w^{p_{k}}}=x_{k}$.
\end{lemma}
\begin{lemma}\label{normofw}
$w\in A_{p_{1}, p_{2}}^{Q}$.
\end{lemma}
\begin{lemma}\label{ravenstvoBx}
$|\{w\geqslant 1\}|=B(x).$
\end{lemma}
We prove the Lemma \ref{normofw} in the Section \ref{calcofthenorm}. Let us prove Lemmata \ref{koordinatiw} and \ref{ravenstvoBx}.
\begin{proof}[Proof of the Lemma \ref{koordinatiw}]
For $k=1$ we perform a direct calculation:
\begin{multline}
\ave{w^{p_{1}}}=\frac{\gamma_{-}^{p_{1}}-v_{-}^{p_{1}}}{1-v_{-}^{p_{1}}}a+v_{-}^{p_{1}}\frac{1-\gamma_{-}^{p_{1}}}{1-v_{-}^{p_{1}}}a+v_{-}^{p_{1}}a^{\nu p_{1}}\frac{1}{1-\nu p_{1}}(1-a^{1-\nu p_{1}})=\\=\gamma_{-}^{p_{1}}a+v_{-}^{p_{1}}a^{\nu p_{1}}\frac{1}{1-\nu p_{1}}-v_{-}^{p_{1}}a\frac{1}{1-\nu p_{1}}=
\gamma_{-}^{p_{1}}a+v^{p_{1}}\gamma_{+}^{p_{1}}-v_{-}^{p_{1}}\gamma_{+}^{p_{1}}a=\\=\gamma_{-}^{p_{1}}a+x_{1}-\gamma_{-}^{p_{1}}a=x_{1}.
\end{multline}
For $k=2$ we need the following:
\begin{equation}\label{vazhnoeravenstvo}
\frac{1}{1-\nu p_{2}}=Q^{-p_{2}}\gamma_{+}^{p_{2}}.
\end{equation}
To prove it take equation \eqref{gamma}:
$$
Q^{-p_{2}}\left(1-\frac{p_{2}}{p_{1}}\right)\gamma_{+}^{p_{2}}=
1-\frac{p_{2}}{p_{1}}Q^{-p_{2}}\gamma_{+}^{p_{2}-p_{1}}.
$$
Multiplying it by $Q^{p_{2}}\gamma_+^{-p_{2}}$ we get:
$$
1-\frac{p_{2}}{p_{1}}=Q^{p_{2}}\gamma_{+}^{-p_{2}}-\frac{p_{2}}{p_{1}}\gamma_{+}^{-p_{1}}=Q^{p_{2}}\gamma_{+}^{-p_{2}}-\frac{p_{2}}{p_{1}}(1-\nu p_{1}),
$$
so
$$
Q^{p_{2}}\gamma_{+}^{-p_{2}}=1-\frac{p_{2}}{p_{1}}+\frac{p_{2}}{p_{1}}-\nu p_{2}=1-\nu p_{2},
$$
which is what we need.

Observe also that points $(1,1)$, $(\gamma_{-}^{p_{1}}, Q^{-p_{2}}\gamma_{-}^{p_{2}})$ and $(v_{-}^{p_{1}}, v_{-}^{p_{2}})$ lie on a line $\ell_{-}$. Therefore, we have the following equation:
$$
\frac{\gamma_{-}^{p_{1}}-v_{-}^{p_{1}}}{1-v_{-}^{p_{1}}}=\frac{Q^{-p_{2}}\gamma_{-}^{p_{2}}-v_{-}^{p_{2}}}{1-v_{-}^{p_{2}}}.
$$
Now the calculation for $\ave{w^{p_{2}}}$ is exactly the same as for $\ave{w^{p_{1}}}$.
\end{proof}
The equation \eqref{vazhnoeravenstvo} is very useful for us, so we want to put it as lemma.
\begin{lemma}
In our notation, we have
$$
\frac{1}{1-\nu p_{1}}=\gamma_{+}^{p_{1}},
$$
$$
\frac{1}{1-\nu p_{2}}=Q^{-p_{2}}\gamma_{+}^{p_{2}}.
$$
Consequently,
$$
x_{k}=\frac{v^{p_{k}}}{1-\nu p_{k}}.
$$
\end{lemma}
\begin{proof}[Proof of the Lemma \ref{ravenstvoBx}]
Since $w$ is a decreasing function and $v_{-}<1$, we get
$$
|\{w\geqslant 1\}|=\frac{\gamma_{-}^{p_{1}}-v_{-}^{p_{1}}}{1-v_{-}^{p_{1}}}\cdot a.
$$
Let us count $B(x)$. It will be a direct calculation. Precisely, substituting $x_{k}=\frac{v^{p_{k}}}{1-\nu p_{k}}$ into \eqref{answerforb}, we get
\begin{equation}
\begin{split}
B(x)&=\frac{1}{1-A}\frac{v_{-}^{-\frac{(p_{1}-p_{2})A}{1-A}}}{1-v_{-}^{p_{1}}}v^{\frac{p_{1}-p_{2}}{1-A}} \left(\frac{p_{1}-p_{2}}{p_{2}}v^{p_{2}}+\frac{1}{1-\nu p_{1}}v^{p_{2}}-\frac{p_{1}}{p_{2}}\frac{1}{1-\nu p_{2}}v^{p_{2}}\right)=\\
&=\frac{1}{1-A}\frac{v_{-}^{-\frac{(p_{1}-p_{2})A}{1-A}}}{1-v_{-}^{p_{1}}}v^{\frac{p_{1}-p_{2}A}{1-A}}\left(\frac{p_{1}}{p_{2}}-1+\frac{1}{1-\nu p_{1}}-\frac{p_{1}}{p_{2}}\frac{1}{1-\nu p_{2}}\right)=\\&=
\frac{1}{1-A}\frac{v_{-}^{-\frac{(p_{1}-p_{2})A}{1-A}}}{1-v_{-}^{p_{1}}}v^{\frac{p_{1}-p_{2}A}{1-A}}\left(\frac{\nu p_{1}}{1-\nu p_{1}}-\frac{p_{1}}{p_{2}}\frac{\nu p_{2}}{1-\nu p_{2}}\right)=\\&=
\frac{1}{1-A}\frac{v_{-}^{-\frac{(p_{1}-p_{2})A}{1-A}}}{1-v_{-}^{p_{1}}}v^{\frac{p_{1}-p_{2}A}{1-A}}\left(\frac{\nu p_{1}}{1-\nu p_{1}}-\frac{\nu p_{1}}{1-\nu p_{2}}\right)=\\&=\frac{1}{1-A}\frac{v_{-}^{-\frac{(p_{1}-p_{2})A}{1-A}}}{1-v_{-}^{p_{1}}}v^{\frac{p_{1}-p_{2}A}{1-A}}\nu p_{1} \frac{\nu p_{1}-\nu p_{2}}{(1-\nu p_{1})(1-\nu p_{2})}.
\end{split}
\end{equation}
Recall that $A=Q^{-p_{2}}\gamma_{+}^{p_{2}-p_{1}}$, so
$$
1-A=1-\gamma_{+}^{-p_{1}}\cdot Q^{-p_{2}}\gamma_{+}^{p_{2}}=1-\frac{1-\nu p_{1}}{1-\nu p_{2}}=\frac{\nu p_{1}-\nu p_{2}}{1-\nu p_{2}}.
$$
Therefore,
$$
B(x)=\frac{v_{-}^{-\frac{(p_{1}-p_{2})A}{1-A}}}{1-v_{-}^{p_{1}}}v^{\frac{p_{1}-p_{2}A}{1-A}}\nu p_{1} \frac{1}{(1-\nu p_{1})}.
$$
Moreover, observe that
$$
v_{-}^{-\frac{(p_{1}-p_{2})A}{1-A}}=v_{-}^{-\frac{p_{1}-p_{2}A}{1-A}}\cdot v_{-}^{p_{1}},
$$
so
$$
B(x)=\frac{a^{\nu\cdot \frac{p_{1}-p_{2}A}{1-A}}}{1-v_{-}^{p_{1}}} v_{-}^{p_{1}}\frac{\nu p_{1}}{1-\nu p_{1}}=\frac{a^{\nu\cdot \frac{p_{1}-p_{2}A}{1-A}}}{1-v_{-}^{p_{1}}} \frac{\gamma_{-}^{p_{1}}}{\gamma_{+}^{p_{1}}}\cdot \nu p_{1}\cdot \gamma_{+}^{p_{1}}=\frac{\nu p_{1} \gamma_{-}^{p_{1}}}{1-v_{-}^{p_{1}}} \cdot a^{\nu \frac{p_{1}-p_{2}A}{1-A}}.
$$
One more time recall that $A=Q^{-p_{2}}\gamma_{+}^{p_{2}-p_{1}}=\frac{1-\nu p_{1}}{1-\nu p_{2}}$, so
$$
p_{1}-p_{2}A=p_{1}-p_{2}\frac{1-\nu p_{1}}{1-\nu p_{2}}=\frac{p_{1}-p_{2}}{1-\nu p_{2}},
$$
thus
$$
\frac{p_{1}-p_{2}A}{1-A}=\frac{1}{\nu}.
$$
Using that, we get
$$
B(x)=\frac{\nu p_{1}\cdot \gamma_{-}^{p_{1}}}{1-v_{-}^{p_{1}}}a=\frac{\gamma_{-}^{p_{1}}+(\nu p_{1}-1)\gamma_{-}^{p_{1}}}{1-v_{-}^{p_{1}}}a=\frac{\gamma_{-}^{p_{1}}-\gamma_{+}^{-p_{1}}\gamma_{-}^{p_{1}}}{1-v_{-}^{p_{1}}}a=\frac{\gamma_{-}^{p_{1}}-v_{-}^{p_{1}}}{1-v_{-}^{p_{1}}}a,
$$
and that is exactly what we want to get.
\end{proof}

\subsubsection{The case of arbitrary $x\in \Om_{\textup{IV}}$}
We now take an $x\in \Om_{\textup{IV}}$ and a point $(v^{p_{1}}, v^{p_{2}})\in \Gamma$ such that $x\in \ell_{+}(v)$. We take a point $y=(\gamma_{+}^{p_{1}}v^{p_{1}}, Q^{-p_{2}}\gamma_{+}^{p_{2}}v^{p_{2}})$. Assume that $w_{y}$ is a function that we have built in the previous subsection. Note that there is a number $\lambda\in[0,1]$ such that
\begin{align*}
x_{1}&=(1-\lambda) v^{p_{1}}+\lambda v^{p_{1}}\gamma_{+}^{p_{1}},\\
x_{2}&=(1-\lambda) v^{p_{2}}+\lambda Q^{-p_{2}}v^{p_{2}}\gamma_{+}^{p_{2}}.
\end{align*}
Denote now
$$
w(t)=\begin{cases} w_{y}\left(\frac{t}{\lambda}\right), &t\in [0, \lambda], \\
                  v, &t\in (\lambda, 1]. \end{cases}
$$
Take a function $w_y^\lambda (t)=w_y\left(\frac{t}{\lambda}\right)$. It is defined when $t\leqslant \lambda$, but when $t$ is close to $\lambda$, it is a power function, so we can extend it to the interval $[0,1]$. So we assume now that our $w_y\left(\frac{t}{\lambda}\right)$ is defined for $t\in [0,1]$.
We note that
$$
w(t)=\begin{cases} w_{y}\left(\frac{t}{\lambda}\right), &w_{y}\left(\frac{t}{\la}\right)\geqslant v, \\
                  v, &w_{y}\left(\frac{t}{\la}\right)\leqslant v. \end{cases}
$$
Therefore, by the Lemma \ref{cutoff}, $w\in A_{p_{1}, p_{2}}^{Q}$.

Moreover, since $B$ is linear on $\ell_{+}(v)$ and since $v<1$, we get
$$
B(x)=(1-\lambda)B(v^{p_{1}}, v^{p_{2}})+\lambda B(y)=\lambda B(y)=\lambda |\{w_{y}(t)\geqslant 1\}|=|\{w(t)\geqslant 1\}|,
$$
which completely finishes our proof.

\section{Calculating the $A_{p_{1}, p_{2}}$-``norm'' of the test-function}\label{calcofthenorm}
In this section we estimate a ``norm'' of some particular functions. We remind some notation.
We fix a point $x=(x_{1}, x_{2})=(\gamma_{+}^{p_{1}}v^{p_{1}}, Q^{-p_{2}}\gamma_{+}^{p_{2}}v^{p_{2}})\in \Gamma_{Q}$.
First we take a number $\nu$ such that
$$
\frac{1}{1-\nu p_{1}}=\gamma_{+}^{p_{1}}.
$$
We proved that then the following holds:
$$
\frac{1}{1-\nu p_{2}}=Q^{-p_{2}}\gamma_{+}^{p_{2}}.
$$
Consequently,
$$
x_{k}=\frac{v^{p_{k}}}{1-\nu p_{k}}.
$$
We remind that
$$
v_{-}=\frac{\gamma_{-}}{\gamma_{+}}.
$$
and take
$$
a=\left(\frac{v}{v_{-}}\right)^{\frac{1}{\nu}}.
$$
Now we denote
$$
w(t)=\begin{cases}1, &t\in [0, \frac{\gamma_{-}^{p_{1}}-v_{-}^{p_{1}}}{1-v_{-}^{p_{1}}}\cdot a], \\
                  v_{-}, &t\in [\frac{\gamma_{-}^{p_{1}}-v_{-}^{p_{1}}}{1-v_{-}^{p_{1}}}\cdot a, a],\\
                  v_{-}\cdot \left(\frac{a}{t}\right)^{\nu}, &t\in[a,1].
\end{cases}
$$

We introduce a simpler function
$$
u(t)=v_{-}\cdot \left(\frac{a}{t}\right)^{\nu}, \; t\in [0,1].
$$
Our first lemma is the following.
\begin{lemma}\label{shirlimirli}
$$
u\in A_{p_{1}, p_{2}}^{Q}.
$$
\end{lemma}
This lemma was proved in ~\cite{Va}, but we repeat the proof.
\begin{proof}[Proof]
We take an interval $J=[\alpha, \beta]$ and write
\begin{multline}
\av{u^{p_{1}}}{J}^{\frac{1}{p_{1}}}\av{u^{p_{2}}}{J}^{-\frac{1}{p_{2}}}=\left(x_{1}\cdot\frac{\beta^{1-\nu p_{1}}-\alpha^{1-\nu p_{1}}}{\beta-\alpha}\right)^{\frac{1}{p_{1}}}\left(x_{2}\cdot\frac{\beta^{1-\nu p_{2}}-\alpha^{1-\nu p_{2}}}{\beta-\alpha}\right)^{-\frac{1}{p_{2}}}=\\=Q\cdot \left(\frac{\beta^{1-\nu p_{1}}-\alpha^{1-\nu p_{1}}}{\beta-\alpha}\right)^{\frac{1}{p_{1}}}\cdot\left(\frac{\beta^{1-\nu p_{2}}-\alpha^{1-\nu p_{2}}}{\beta-\alpha}\right)^{-\frac{1}{p_{2}}}.
\end{multline}
To prove that the left-hand side is not greater then $Q$ we now have to prove that for every $\alpha$ and $\beta$, such that $0\leqslant\alpha\leqslant\beta\leqslant1$ the following estimate is true:
$$
\left(\frac{\beta^{1-\nu p_{1}}-\alpha^{1-\nu p_{1}}}{\beta-\alpha}\right)^{\frac{1}{p_{1}}}\cdot\left(\frac{\beta^{1-\nu p_{2}}-\alpha^{1-\nu p_{2}}}{\beta-\alpha}\right)^{-\frac{1}{p_{2}}}\leqslant 1.
$$
Denote $s=\frac{\alpha}{\beta}$. Then the left-hand side of the last expression is equal to
$$
g(s)=\left(\frac{1-s^{1-\nu p_{1}}}{1-s}\right)^{\frac{1}{p_{1}}}\cdot\left(\frac{1-s^{1-\nu p_{2}}}{1-s}\right)^{-\frac{1}{p_{2}}},
$$
where $0\leqslant s\leqslant 1$.

Now we shall stop treating $g$ as a function of $s$, but we consider is as a function of $\nu>0$! Then
$$
\frac{\partial g}{\partial \nu}=\mbox{something positive}\cdot \log(s)\cdot (1-s^{\nu p_1-\nu p_2})\leqslant 0,
$$
and therefore
$$
g(s)\leqslant g(\nu=0)=1,
$$
and we are done.
\end{proof}
Then, as in the Subsection \ref{subsectionapproxim}, we have
$$
\underline{u}_{v_{-}}(t)=\begin{cases}
                  v_{-}, &t\in [0, a],\\
                  v_{-}\cdot \left(\frac{a}{t}\right)^{\nu}, &t\in[a,1].
\end{cases}
$$
From the Lemma \ref{cutoff} we know that $\underline{u}_{v_{-}}\in A_{p_{1}, p_{2}}^{Q}$.
Recall that to prove that the initial function $w$ is also in $A_{p_{1}, p_{2}}^{Q}$, we should prove that for every interval $J\subset[0,1]$
$$
\av{w^{p_{1}}}{J}^{\frac{1}{p_{1}}}\av{w^{p_{2}}}{J}^{-\frac{1}{p_{2}}}\leqslant Q.
$$
But from Lemmata \ref{shirlimirli} and \ref{kusochnaya} we already know it for many intervals $J$. Consequently, we should check the last inequality for intervals $J=[\alpha, \beta]$ such that $\alpha <\frac{\gamma_{-}^{p_{1}}-v_{-}^{p_{1}}}{1-v_{-}^{p_{1}}}\cdot a$, $\beta>a$. It will be our last step.
\begin{lemma}\label{shirmanirli}
If $J=[\alpha, \beta]$, and $\alpha <\frac{\gamma_{-}^{p_{1}}-v_{-}^{p_{1}}}{1-v_{-}^{p_{1}}}\cdot a$, $\beta>a$, then
$$
\av{w^{p_{1}}}{J}^{\frac{1}{p_{1}}}\av{w^{p_{2}}}{J}^{-\frac{1}{p_{2}}}\leqslant Q.
$$
\end{lemma}
\begin{proof}[Proof]
Obviously,
\begin{multline}
\av{w^{p_{1}}}{J}=\\=\frac{1}{\beta-\alpha}\cdot \left[\Big(\frac{\gamma_{-}^{p_{1}}-v_{-}^{p_{1}}}{1-v_{-}^{p_{1}}}\cdot a - \alpha\Big) + v_{-}^{p_{1}}\cdot a\cdot\frac{1-\gamma_{-}^{p_{1}}}{1-v_{-}^{p_{1}}} + \frac{v^{p_{1}}}{1-\nu p_{1}} \Big(\beta^{1-\nu p_{1}}-a^{1-\nu p_{1}}\Big) \right].
\end{multline}
Note that
$$
x_{1}=\ave{w^{p_{1}}}=\frac{\gamma_{-}^{p_{1}}-v_{-}^{p_{1}}}{1-v_{-}^{p_{1}}}\cdot a+v_{-}^{p_{1}}\cdot a\cdot\frac{1-\gamma_{-}^{p_{1}}}{1-v_{-}^{p_{1}}} + \frac{v^{p_{1}}}{1-\nu p_{1}} \Big(1-a^{1-\nu p_{1}}\Big),
$$
and
$$
x_{1}=\frac{v^{p_{1}}}{1-\nu p_{1}}.
$$
Therefore,
$$
\av{w^{p_{1}}}{J}=\frac{x_{1}\beta^{1-\nu p_{1}}-\alpha}{\beta-\alpha},
$$
and, similarly,
$$
\av{w^{p_{2}}}{J}=\frac{x_{2}\beta^{1-\nu p_{2}}-\alpha}{\beta-\alpha}.
$$
Therefore,
$$
\av{w^{p_{1}}}{J}^{\frac{1}{p_{1}}}\av{w^{p_{2}}}{J}^{-\frac{1}{p_{2}}}=\left(\frac{x_{1}\beta^{1-\nu p_{1}}-\alpha}{\beta-\alpha}\right)^{\frac{1}{p_{1}}}\cdot\left(\frac{x_{2}\beta^{1-\nu p_{2}}-\alpha}{\beta-\alpha}\right)^{-\frac{1}{p_{2}}}.
$$
Let
$$
F(\alpha, \beta)=\left(\frac{x_{1}\beta^{1-\nu p_{1}}-\alpha}{\beta-\alpha}\right)^{\frac{1}{p_{1}}}\cdot\left(\frac{x_{2}\beta^{1-\nu p_{2}}-\alpha}{\beta-\alpha}\right)^{-\frac{1}{p_{2}}}
$$
be the right-hand side of the last expression. First, we introduce new variables
$$
t=\frac{x_{1}}{\beta^{\nu p_{1}}},
$$
$$
s=\frac{\alpha}{\beta}.
$$
Denote
$$
G(s,t)=F(\alpha,\beta)=\left(\frac{t-s}{1-s}\right)^{\frac{1}{p_{1}}} \cdot \left(\frac{Q^{-p_{2}}t^{\frac{p_{2}}{p_{1}}}-s}{1-s}\right)^{-\frac{1}{p_{2}}}.
$$
We prove the following lemma.
\begin{lemma}
$G(s,t)$ does not attain its maximum in the interior of its domain.
\end{lemma}
\begin{zamech}
We have no intention to write the domain of $G$ explicitly. However, its domain has some obvious properties. For example, $0\leqslant s <1$ and
$$
\left(\frac{t-s}{1-s}\right)=\frac{x_{1}\beta^{1-\nu p_{1}}-\alpha}{\beta-\alpha}=\av{w^{p_{1}}}{J}>0.
$$
must be separated from zero.
\end{zamech}
\begin{proof}[Proof of the Lemma]
$G$ is a smooth function, so if it has a maximum in the interior of its domain, then at this point both $G^{\prime}_{t}$ and $G^{\prime}_{s}$ are equal to zero.
Denote
$$
M=\frac{t-s}{1-s},
$$
$$
N=\frac{Q^{-p_{2}}t^{\frac{p_{2}}{p_{1}}}-s}{1-s}.
$$
Then
$$
M^{\prime}_{t}=\frac{1}{1-s},
$$
$$
N^{\prime}_{t}=\frac{p_{2}}{p_{1}}Q^{-p_{2}}t^{\frac{p_{2}}{p_{1}}-1}\frac{1}{1-s}.
$$
Therefore, $G^{\prime}_{t}=0$ if and only if
$$
\frac{N}{p_{1}}-\frac{M}{p_{1}}Q^{-p_{2}}t^{\frac{p_{2}}{p_{1}}-1}=0,
$$
which yields
$$
\left(Q^{-p_{2}}t^{\frac{p_{2}}{p_{1}}}-s\right)-(t-s)Q^{-p_{2}}t^{\frac{p_{2}}{p_{1}}-1}=0.
$$
and, therefore,
$$
s(Q^{-p_{2}}t^{\frac{p_{2}}{p_{1}}-1}-1)=0.
$$
Since in the interior of the domain $s>0$, we get $Q^{-p_{2}}t^{\frac{p_{2}}{p_{1}}-1}=1$.
Note that then $t\not=1$.

Now let us count the partial derivative with respect to $s$, assuming that the last equality holds.
Obviously,
$$
M^{\prime}_{s}=\frac{t-1}{(s-1)^{2}},
$$
$$
N^{\prime}_{s}=\frac{Q^{-p_{2}}t^{\frac{p_{2}}{p_{1}}}-1}{(s-1)^{2}}=M^{\prime}_{s}.
$$
Observe that $t\not=1$, so
$$
G^{\prime}_{s}=\frac{1}{p_{1}}M^{\frac{1}{p_{1}}-1}N^{-\frac{1}{p_{2}}} M^{\prime}_{s} - \frac{1}{p_{2}}M^{\frac{1}{p_{1}}}N^{-\frac{1}{p_{2}}-1}N^{\prime}_{s}.
$$
If $G^{\prime}_{s}=0$, then
$$
\frac{N}{p_{1}}-\frac{M}{p_{2}}=0,
$$
$$
\frac{1}{p_{1}} \frac{Q^{-p_{2}}t^{\frac{p_{2}}{p_{1}}}-s}{1-s} - \frac{1}{p_{2}}\frac{t-s}{1-s}=0.
$$
and since we consider a concrete $t$,
$$
t-s=0,
$$
$$
t=s.
$$
But it contradicts with the second property of our domain from the remark, which finishes the proof.
\end{proof}
Note that our change of variables is obviously an open map. Therefore, the interior of the domain of $F$ maps onto the interior of the domain of $G$ and thus $F$ does not attain its maximum in the interior of its domain.

Let us study $F$ on the boundary of its domain.
\paragraph{Case $\alpha=0$}
Here everything is pretty obvious, because, independently of $\beta$,
$$
F(0, \beta)=Q.
$$

\paragraph{Case $\beta=a$}
Here everything is also easy, since then the third piece of $w$ is not involved and, therefore, we have a linear combination of $(1,1)$ and $(v_{-}^{p_{1}}, v_{-}^{p_{2}})$.

\paragraph{Case $\alpha=\frac{\gamma_{-}^{p_{1}}-v_{-}^{p_{1}}}{1-v_{-}^{p_{1}}}a$}
This case is already done, since here the first piece of $w$ is not involved, and we get a cut-off of a function $u$ from the Lemma \ref{shirlimirli}.

\paragraph{Case $\beta=1$}

This case is more complicated and needs to be studied. Here we have
$$
\av{w^{p_{1}}}{J}^{\frac{1}{p_{1}}}\av{w^{p_{2}}}{J}^{-\frac{1}{p_{2}}}=\left(\frac{x_{1}-\alpha}{1-\alpha}\right)^{\frac{1}{p_{1}}}\cdot\left(\frac{x_{2}-\alpha}{1-\alpha}\right)^{-\frac{1}{p_{2}}},
$$
where
$$
0\leqslant \alpha < \frac{\gamma_{-}^{p_{1}}-v_{-}^{p_{1}}}{1-v_{-}^{p_{1}}}a.
$$
We denote
$$
H(x_{1}, \alpha)=\left(\frac{x_{1}-\alpha}{1-\alpha}\right)^{\frac{1}{p_{1}}}\cdot\left(\frac{x_{2}-\alpha}{1-\alpha}\right)^{-\frac{1}{p_{2}}}.
$$
Recall that $x_{2}=Q^{-p_{2}}x_{1}^{\frac{p_{2}}{p_{1}}}$. We prove the following.
\begin{lemma}
The following is true:
$$
\sig(H^{\prime}_{x_{1}})=\sig(p_{1}).
$$
\end{lemma}
Suppose we have proved the lemma. First we show, how to finish the proof of the theorem.
We note that $x_{1}\leqslant \gamma_{-}^{p_{1}}$ if $p_{1}>0$ and $x_{1}\geqslant \gamma_{-}^{p_{1}}$ if $p_{1}<0$. Therefore,
$$
H(x_{1}, \alpha)\leqslant H(\gamma_{-}^{p_{1}}, \alpha).
$$
We would like to estimate the right-hand side. Assume, therefore, that $x_{1}=\gamma_{-}^{p_{1}}$, $x_{2}=Q^{-p_{2}}\gamma_{-}^{p_{2}}$.
We introduce a function
$$
q(t)=\begin{cases}1, &t\in [0, \frac{\gamma_{-}^{p_{1}}-v_{-}^{p_{1}}}{1-v_{-}^{p_{1}}}] \\
                v_{-}, &t\in [\frac{\gamma_{-}^{p_{1}}-v_{-}^{p_{1}}}{1-v_{-}^{p_{1}}}, 1].\end{cases}
$$
Note that $\ave{q^{p_{1}}}=\gamma_{-}^{p_{1}}=x_{1}$, $\ave{q^{p_{2}}}=Q^{-p_{2}}\gamma_{-}^{p_{2}}=x_{2}$ and for every $\alpha\leqslant \frac{\gamma_{-}^{p_{1}}-v_{-}^{p_{1}}}{1-v_{-}^{p_{1}}}$ we have  $\av{q^{p_{k}}}{[\alpha, 1]}=\frac{x_{k}-\alpha}{1-\alpha}$.
Since the whole line interval, connecting $(1,1)$ and $(v_{-}^{p_{1}}, v_{-}^{p_{2}})$ lies in $\Omega$, we have
$$
\left(\frac{x_{1}-\alpha}{1-\alpha}\right)^{\frac{1}{p_{1}}}\left(\frac{x_{2}-\alpha}{1-\alpha}\right)^{-\frac{1}{p_{2}}}\leqslant Q
$$
for every $\alpha\leqslant \frac{\gamma_{-}^{p_{1}}-v_{-}^{p_{1}}}{1-v_{-}^{p_{1}}}$.

But to estimate $H$ we need to consider $\alpha\leqslant \frac{\gamma_{-}^{p_{1}}-v_{-}^{p_{1}}}{1-v_{-}^{p_{1}}}\cdot a$, which is stronger then the previous one, since $a\leqslant 1$. Therefore,
$$
H(\gamma_{-}^{p_{1}}, \alpha)\leqslant Q,
$$
and that is it.
\end{proof}
It remains to prove the last Lemma.
\begin{proof}[Proof]
Recall that we want to calculate $H^{\prime}_{x_{1}}(x_{1}, \alpha)$. Recall that
$$
x_{2}=Q^{-p_{2}}x_{1}^{\frac{p_{2}}{p_{1}}}.
$$
Therefore,
$$
\frac{dx_{2}}{dx_{1}}=\frac{p_{2}}{p_{1}}\frac{x_{2}}{x_{1}}.
$$
Denoting $M=\frac{x_{1}-\alpha}{1-\alpha}$, $N=\frac{x_{2}-\alpha}{1-\alpha}$, we get:
\begin{align*}
\frac{\partial H}{\partial x_{1}}=&\frac{1}{p_{1}}M^{\frac{1}{p_{1}}-1}N^{-\frac{1}{p_{2}}} - \frac{1}{p_{2}}M^{\frac{1}{p_{1}}}N^{-\frac{1}{p_{2}}-1} \frac{p_{2}x_{2}}{p_{1}x_{1}}=\\
=&M^{\frac{1}{p_{1}}-1}N^{-\frac{1}{p_{2}}-1}\left(\frac{1}{p_{1}}\frac{x_{2}-\alpha}{1-\alpha}-\frac{1}{p_{2}}\frac{x_{1}-\alpha}{1-\alpha}\frac{p_{2}}{p_{1}}\frac{x_{2}}{x_{1}}\right)
=\\
=&M^{\frac{1}{p_{1}}-1}N^{-\frac{1}{p_{2}}-1} \frac{\alpha}{1-\alpha}\frac{1}{x_{1}}\frac{x_{2}-x_{1}}{p_{1}}.
\end{align*}
All we need now is to prove that $x_{2}-x_{1}\geqslant0$. One can do it algebraically, but in fact it is a pure geometry. It can be seen from the picture, but we give an independent proof.
Notice that $x_{2}\geqslant x_{1}$ if and only if $Q^{-p_{2}}\geqslant x_{1}^{\frac{p_{1}-p_{2}}{p_{1}}}$, which is true if and only if
$$
p_{1}Q^{\frac{p_{1}p_{2}}{p_{2}-p_{1}}}\geqslant p_{1}x_{1}.
$$

Recall that
$$
Q^{-p_{2}}\left(1-\frac{p_{2}}{p_{1}}\right)\gamma^{p_{2}}=
1-\frac{p_{2}}{p_{1}}Q^{-p_{2}}\gamma^{p_{2}-p_{1}}.
$$
We have the following chain:
\begin{multline}
Q^{\frac{p_{1}p_{2}}{p_{2}-p_{1}}}\geqslant \gamma_{-}^{p_{1}} \Leftrightarrow
Q^{p_{2}}\cdot p_{1}\leqslant \gamma_{-}^{p_{2}-p_{1}} \cdot p_{1}\Leftrightarrow \\
p_{1}Q^{-p_{2}}\gamma_{-}^{p_{2}-p_{1}}\geqslant p_{1} \Leftrightarrow
p_{1}\cdot\frac{p_{1}}{p_{2}}\left(1-Q^{-p_{2}}(1-\frac{p_{2}}{p_{1}})\gamma_{-}^{p_{2}}\right)\geqslant p_{1} \Leftrightarrow \\
p_{1}\cdot\left(\frac{p_{1}}{p_{2}}-1\right) \geqslant \frac{p_{1}}{p_{2}}(p_{1}-p_{2})Q^{-p_{2}}\gamma_{-}^{p_{2}} \Leftrightarrow
\frac{p_{1}}{p_{2}}\geqslant \frac{p_{1}}{p_{2}}Q^{-p_{2}}\gamma_{-}^{p_{2}}\Leftrightarrow p_{1}\cdot \frac{Q^{p_{2}}-\gamma_{-}^{p_{2}}}{p_{2}}\geqslant 0.
\end{multline}
since $Q>\gamma_{-}$ we get that
$$
Q^{\frac{p_{1}p_{2}}{p_{2}-p_{1}}}\geqslant \gamma_{-}^{p_{1}} \Leftrightarrow p_{1}\geqslant 0.
$$
Therefore,
$$
p_{1}Q^{\frac{p_{1}p_{2}}{p_{2}-p_{1}}}\geqslant p_{1}\gamma_{-}^{p_{1}}.
$$
But we know that $p_{1}\gamma_{-}^{p_{1}}\geqslant p_{1}x_{1}$, which finishes our proof.
\end{proof}
The whole proof of the Lemma \ref{shirmanirli} is finished, which finishes the proof of the Lemma \ref{normofw}.
\section{An illustration: the $A_2$ case and Reverse H\"{o}lder property}\label{illustr}
This section has two goals. The first one is to write an answer for the Bellman function in one particular case: $p_1=1$, $p_2=-1$. This case is interesting because it corresponds to the $A_2$ condition, which plays leading role in the theory of singular integral operators. This case is also interesting because here we can write an explicit answer in terms of the $A_2$-``norm'' of the weight, avoiding all implicit equation.

The second goal of this section is to show the following statement.
\begin{theorem}
Suppose $w\in A_2$ and $[w]_2=\sup_I \av{w}{I}\av{w^{-1}}{I}=Q$. Then there exist a constant $\alpha_0>0$, depending only on $Q$, such that for every $\alpha$, $0<\alpha<\alpha_0$, the following inequality holds:
$$
\ave{w^{1+\alpha}}\leqslant C\ave{w}^{1+\alpha},
$$
where $C=C(\alpha)$ is a constant, which does not depend on $w$.
\end{theorem}
We refer one more time to the paper ~\cite{DiWa}, where the opposite question was considered: a Reverse H\"{o}lder weight self-improves to an $A_p$ weight.

We should say that this result is known. It was proved, for example, in ~\cite{Va2} with a sharp constant $C$. However, here we prove it just as an application of our sharp estimate for a distribution function of $A_2$ weights.

Let us start calculating the function $B$. We remind the reader that in our case
$$
B(x_1, x_2)=\sup\Big\{ |\{t\colon w(t)\geqslant 1\}|\colon \ave{w}=x_1, \ave{w^{-1}}=x_2, w\in A_{2}^Q\Big\}.
$$
We start with calculating constants $\gamma_{\pm}$. We have an equation
$$
Q(1+1)\frac{1}{\gamma}=1+Q\cdot \frac{1}{\gamma^2},
$$
which has two solutions
$$
\gamma_+=Q+\sqrt{Q^2-Q}, \; \; \; \; \; \gamma_-=Q-\sqrt{Q^2-Q}.
$$
Therefore,
$$
v_- = \frac{\gamma_-}{\gamma_+}=\frac{Q-\sqrt{Q^2-Q}}{Q+\sqrt{Q^2+Q}}.
$$

We know that in $\Om_I$ our function $B=1$. Let us calculate numbers $a,b,c$ for $\Om_{\textup{II}}$.

We have
\begin{align}
&1-v_-=\frac{2\sqrt{Q^{2}-Q}}{Q+\sqrt{Q^2-Q}}, \\
&v_{-}^{-1}-1=\frac{2\sqrt{Q^2-Q}}{Q-\sqrt{Q^2-Q}},\\
&v_- - v_-^{-1} = -4\sqrt{Q^2-Q}.
\end{align}
Further,
\begin{align}
&a=\frac{v_-}{(1-v_-)(v_- - v_-^{-1})}=-\frac{Q-\sqrt{Q^2-Q}}{8(Q^2-Q)},\\
&b=\frac{v_-^{-1}}{(v_-^{-1}-1)(v_- - v_-^{-1})} = -\frac{Q+\sqrt{Q^2-Q}}{8(Q^2-Q)},\\
&c=1-\frac{1}{(v_- - 1)(v_-^{-1}-1)}=1+\frac{1}{4(Q-1)}.
\end{align}

We proceed to the domain $\Om_{\textup{III}}$. Let us find the parameter $v$ in terms of $x_1$ and $x_2$. We have an equation
$$
\frac{1}{v}(1-x_1)-v(1-x_2)=x_2-x_1.
$$
It is a quadratic equation, and we know that $v=1$ is a root. It is not hard to find the second one. Namely,
$$
v=\frac{1-x_1}{x_2-1}.
$$
Therefore, in $\Om_{\textup{III}}$ we have
$$
B(x)=\frac{x_1 x_2 -1}{x_1+x_2-2}.
$$
We now proceed to the domain $\Om_{\textup{IV}}$. Here we need more work. First of all, we should again find our $v$. Our equation (again quadratic) is the following:
$$
x_2=Q\cdot (-1)\cdot \gamma_{+}^{-2} \cdot v^{-2}\cdot(x_1-v)+\frac{1}{v},
$$
which reduces to
$$
x_2 v^{2}-v(1+\frac{Q}{\gamma_+^{2}})+\frac{Q}{\gamma_+^{2}}x_1=0.
$$
We have two solutions of this equation:
$$
v=\frac{1+\frac{Q}{\gamma_+^{2}}\pm \sqrt{\left(1+\frac{Q}{\gamma_+^{2}}\right)^2-4\frac{Q}{\gamma_{+}^{2}}x_1 x_2}}{2x_2}.
$$
We should take the biggest value of $v$. It is obvious from the picture, since we always want our point $x$ to lie between $(v, v^{-1})$ and the point $(\gamma_+ v, Q(\gamma_+v)^{-1})$ (we could appeal to the definition of $v$ from the general consideration, but we repeat it to make it easier to read).

So, we put
$$
v=\frac{1+\frac{Q}{\gamma_+^{2}}+ \sqrt{\left(1+\frac{Q}{\gamma_+^{2}}\right)^2-4\frac{Q}{\gamma_{+}^{2}}x_1 x_2}}{2x_2}.
$$
We can make the following amazing simplification:
$$
\frac{Q}{\gamma_+^{2}}=Q\frac{\gamma_{-}}{\gamma_+} \frac{1}{\gamma_- \gamma_+}=v_-.
$$
Therefore,
$$
v=\frac{1+v_- + \sqrt{(1+v_-)^{2} - 4v_- x_1 x_2}}{2x_2}.
$$
Note also that we had a notation
$$
A=\frac{Q}{\gamma_+^{2}}=v_-.
$$
Therefore, we get an easy answer for $B$:
$$
B(x)=\frac{1}{1-v_-} \frac{v_-^{-\frac{2v_-}{1-v_-}}}{1-v_-}v^{\frac{2}{1-v_-}} \cdot (x_2 + \frac{x_1}{v^{2}} - \frac{2}{v}) = \frac{1}{1-v_-} \frac{v_-^{-\frac{2v_-}{1-v_-}}}{1-v_-} v^{\frac{2}{1-v_-}-2}\cdot (v^{2}x_2 + x_1 - 2v).
$$
But we have a nice relation between $v$ and $(x_1, x_2)$, namely,
$$
v^{2}x_2=v-v_- (x_1-v).
$$
We get, therefore, that
$$
B(x)=\frac{v_-^{-\frac{2v_-}{1-v_-}}}{1-v_-}\cdot v^{\frac{2v_-}{1-v_-}}\cdot (x_1-v).
$$
Finally,
$$
x_1-v = \frac{2x_1x_2 - (1+v_- + \sqrt{(1+v_-)^{2}-4v_- x_1 x_2})}{2x_2}.
$$
We know that $x_1x_2\leqslant Q$, so
$$
x_1-v \leqslant \frac{2Q - 1-v_- - \sqrt{(1+v_-)^{2} - 4Qv_-}}{2x_2}.
$$
The last expression is equal to
$$
\frac{\sqrt{Q^2-Q}}{x_2},
$$
so we get
$$
x_1-v\leqslant \frac{\sqrt{Q^2-Q}}{x_2}.
$$
Note that this estimate is in some sense sharp. We cannot guarantee that if our $w$ has $A_2$-``norm'' equal to $Q$ then is attains on the initial interval. But there are a lot of such functions, for example, the one from Section \ref{menshe}.

We get the following estimate for $B(x)$, when $x\in \Om_{\textup{IV}}$:
$$
B(x)\leqslant \frac{v_-^{-\frac{2v_-}{1-v_-}}}{1-v_-} v^{\frac{2v_-}{1-v_-}} \frac{\sqrt{Q^{2}-Q}}{x_2}.
$$
Moreover, since $x_1x_2\in [1, Q]$, we get
$$
v\asymp \frac{1}{x_2}.
$$
Therefore,
$$
B(x)\leqslant C(Q) x_2^{-\frac{Q}{\sqrt{Q^2-Q}}}.
$$
We also note that if $x_1x_2=Q$ then
$$
C(Q)=\frac{\gamma_+^{\frac{2v_-}{1-v_-}}}{1-v_-}\sqrt{Q^2-Q}.
$$
Therefore, we write several answers for $B$.

$$
B(x)=\begin{cases} 1, & x\in\Om_{\textup{I}} \\
                -\frac{Q-\sqrt{Q^2-Q}}{8(Q^2-Q)}x_1 - \frac{Q+\sqrt{Q^2-Q}}{8(Q^2-Q)}x_2 + 1 + \frac{1}{4(Q-1)}, &x\in \Om_{\textup{II}}\\
                \frac{x_1x_2-1}{x_1+x_2-2}, & x\in \Om_{\textup{III}}\\ \\
                \frac{v_-^{-\frac{2v_-}{1-v_-}}}{1-v_-}\cdot v^{\frac{2v_-}{1-v_-}}\cdot (x_1-v), &x\in \Om_{\textup{IV}}
                \end{cases}.
$$
To proceed we also write an estimate for $B$:
$$
B(x)\leqslant \begin{cases} 1, &x\in \Om_{\textup{I}}\cup \Om_{\textup{II}}\\
                        \frac{Q-1}{x_1+x_2-2}, &x\in \Om_{\textup{III}}\\
                        C(Q)x_2^{-\frac{Q}{\sqrt{Q^2-Q}}}, &x\in\Om_{\textup{IV}}
                        \end{cases}.
$$

Now we proceed to the Reverse H\"{o}lder property. For the sake of simplicity we consider a function $w$ such that $\ave{w}\ave{w^{-1}}=Q$. It simplifies things a little since then the point $(\ave{tw}, \ave{t^{-1} w^{-1}})$ is on the curve $\Gamma_Q$ and, for example, never comes to $\Om_{\textup{III}}$.

Next, we use that
$$
\ave{w^{1+\alpha}}=(1+\alpha) \ili_{0}^{\infty} s^{\alpha} F_w(s) ds,
$$
where
$$
F_w(s)=|\{t\colon w(t)\geqslant s\}|.
$$
We know that
$$
F_w(s)\leqslant B(x_1, x_2; s) = B(\frac{x_1}{s}, x_2 s),
$$
thus
$$
\ave{w^{1+\alpha}}\leqslant (1+\alpha)\ili_{0}^{\infty} s^{\alpha}B(\frac{x_1}{s}, x_2s)ds.
$$
We consider a point $S=(\frac{x_1}{s}, x_1s)$. Note that
$$
S\in \begin{cases} \Om_{\textup{I}}, &s\leqslant\frac{x_1}{\gamma_+} \\
                    \Om_{\textup{II}}, &s\in [\frac{x_1}{\gamma_{+}}, \frac{x_1}{\gamma_{-}}]\\
                    \Om_{\textup{IV}}, &s\geqslant \frac{x_1}{\gamma_{-}}
                    \end{cases}.
$$
Thus,
$$
\ave{w^{1+\alpha}}\leqslant (1+\alpha)\ili_{0}^{\infty} s^{\alpha}B(\frac{x_1}{s}, x_2 s)ds\leqslant
C\left( \ili_{0}^{\frac{x_1}{\gamma_-}} s^{\alpha}ds + \ili_{\frac{x_1}{\gamma_{-}}}^{\infty}s^{\alpha} x_{2}^{-\frac{Q}{\sqrt{Q^2-Q}}}s^{-\frac{Q}{\sqrt{Q^{2}-Q}}}ds \right).
$$
Note that the right-hand side gives us an estimate of the form $Cx_1^{1+\alpha} = C\ave{w}^{1+\alpha}$ as long as the second integral converges on $\infty$. It does when
$$
\alpha-\frac{Q}{\sqrt{Q^2-Q}}<-1,
$$
or, equivalently,
$$
\alpha < \sqrt{\frac{Q}{Q-1}}-1.
$$
This finishes our proof.
\section{Some final remarks}
In this section we state some remarks which are about some unconsidered cases.

First of all, we did not consider cases $p_{k}=0, \pm \infty$. However, in these cases our method works in the same way. In the case $p=0$ the expression
$\av{w^p}J^{\frac1p}$ has to be replaced by $\exp\av{\log w}J$. It has to be
replaced by $\sup_J w$ in the case $p=+\infty$ and by $\inf_J w$ in the case
$p=-\infty$. The answer for $\mathcal{B}$ in these cases will be obtained by passing to the limit when $p_{k}\to 0, \pm \infty$.
We also note that for the $A_\infty$ case, i.e., when $p_1=1$ and $p_2=0$, one can get an answer without passing to limit. The correct variables will be $x_1=\ave{w}$ and $x_2=\ave{\log w}$ with the relation
$$
x_1\exp(-x_2)\in [1, Q].
$$
With the same splitting of the domain as it was above, the answer is the following:
$$
B(x)=\begin{cases} 1, & x\in \Om_{\textup{I}}\\
-\frac{v_-}{(v_- - 1)^2}x_1 + \frac{1}{\log (v_-)}\cdot \frac{1}{v_- - 1} x_2 + \left(1+\frac{v_-}{(v_- - 1)^2}\right), &x\in \Om_{\textup{II}} \\
\frac{x_1 - v}{1-v}, &x\in \Om_{\textup{III}}\\
\frac{\gamma_+}{\gamma_+ - 1}\cdot \frac{1}{1-v_-}\cdot \left(x_1 - x_2 v - v(1-\log(v))\right), &x\in \Om_{\textup{IV}}
\end{cases}.
$$
where the function $v$ is defined by an implicit formula:
$$
v(x_1, x_2)=\begin{cases}x_2(1-v)=( 1-x_1) \cdot \log(v), &x\in \Om_{\textup{III}}\\
vx_2 = \frac{1}{\gamma_+}\cdot(x_1-v) + v\log(v), &x\in \Om_{\textup{IV}}
\end{cases}.
$$

Further, we can consider another Bellman function, when we calculate
$$\sup\left(|\{w>1\}|, \ldots \right).$$ The answer will be absolutely the same in every points except $(1,1)$, where our new function will be zero. Also we will not have an extremal function for every point; instead, we need to build an extremal sequence.

Moreover, since
$$
|\{w\leqslant 1\}|=1-|\{w>1\}|,
$$
we get that
$$
\sup\left(|\{w\leqslant 1\}|, \ldots \right)=1-\inf\left( |\{w>1\}|, \ldots \right).
$$
Using our technique, one can easily calculate the right-hand side. The function for $\inf$ can be calculated in the same way as $\mathcal{B}$ with one difference: it must be convex instead of being concave.

\end{document}